\documentclass[10pt]{amsart}
\usepackage{color, graphicx}
\usepackage[square,sort,comma,numbers]{natbib}
\numberwithin{equation}{section} \overfullrule 5pt
\usepackage{amsmath,amssymb}

\newdimen\tableauside\tableauside=1.0ex
\newdimen\tableaurule\tableaurule=0.4pt
\newdimen\tableaustep
\def\phantomhrule#1{\hbox{\vbox to0pt{\hrule height\tableaurule
width#1\vss}}}
\def\phantomvrule#1{\vbox{\hbox to0pt{\vrule width\tableaurule
height#1\hss}}}
\def\sqr{\vbox{%
  \phantomhrule\tableaustep

\hbox{\phantomvrule\tableaustep\kern\tableaustep\phantomvrule\tableaustep}%
  \hbox{\vbox{\phantomhrule\tableauside}\kern-\tableaurule}}}
\def\squares#1{\hbox{\count0=#1\noindent\loop\sqr
  \advance\count0 by-1 \ifnum\count0>0\repeat}}
\def\tableau#1{\vcenter{\offinterlineskip
  \tableaustep=\tableauside\advance\tableaustep by-\tableaurule
  \kern\normallineskip\hbox
    {\kern\normallineskip\vbox
      {\gettableau#1 0 }%
     \kern\normallineskip\kern\tableaurule}%
  \kern\normallineskip\kern\tableaurule}}
\def\gettableau#1 {\ifnum#1=0\let\next=\null\else
  \squares{#1}\let\next=\gettableau\fi\next}

\newcommand{\be}{\begin{equation}}
\newcommand{\ee}{\end{equation}}
\newcommand{\bea}{\begin{eqnarray}}
\newcommand{\eea}{\end{eqnarray}}
\newcommand{\ba}{\begin{array}}
\newcommand{\ea}{\end{array}}

\newcommand{\id}{\hbox{1\kern-.27em l}}

\newcommand{\al}{\alpha}

\newcommand{\bet}{\beta}

\newcommand{\la}{\lambda}

\newcommand{\cN}{\mathcal{N}}

\newcommand{\non}{\nonumber}

\newcommand{\SO}{\mathrm{SO}}

\newcommand{\Sp}{\mathrm{Sp}}

\newcommand{\Spin}{\mathrm{Spin}}

\newtheorem{Pro}{Proposition}

\newtheorem{Def}{Definition}

\theoremstyle{definition}
\allowdisplaybreaks

\title{Construction of the Symbol Invariant of Partition}
\author[B. Shou]{B. Shou$^{\dag}$}
\author[Qiao Wu]{Qiao Wu$^*$}

\address{$^*$ College of Logistics and E-commerce \\Zhejiang Wanli University\\
No.8 South Qianhu Road, Ningbo 315100, P.R.China}
\address{$\dag$Center  of Mathematical  Sciences\\
Zhejiang University \\
Hangzhou 310027,  China}

\email{$^{\dag}$ bsoul@zju.edu.cn, \quad $^*$10920005@zju.edu.cn}



\begin{document}
\begin{abstract} 
Symbol is used to describe the Springer correspondence for the classical groups. We prove two structure theorems of symbol.  We propose  a   construction of the symbol of the  rigid partitions in the  $B_n$, $C_n$, and $D_n$ theories.  This construction   is natural and  consists of two basic building blocks.   Using this construction,  we give  closed formulas of symbols for the  rigid partitions in the  $B_n, C_n$, and $D_n$ theories.  One part of the closed formula   is    universal and other  parts are determined by the specific theory.   A comparison of between   this closed formula and the old one is made. Previous  results can be    illustrated  more clearly by this  closed formula.
\end{abstract}
\maketitle

\tableofcontents

\section{Introduction}
A partition $\la$ of the positive integer $n$ is a decomposition $\sum_{i=1}^l \la_i = n$  ($\la_1\ge \la_2 \ge \cdots \ge \la_l$).  The integer $l$  is called the length of the partition. Partitions can be graphically visualized with Young diagrams.  For instance the partition $3^22^31$  corresponds to
\be
\tableau{2 5 7}\non
\ee
Young diagrams  occur in a number of branches of mathematics and physics, including the study of combinatorics, representation theory. They  are also useful tools for the construction of  the eigenstates of Hamiltonian System \cite{Shou:2011} \cite{{Shou:2014}} \cite{{Shou:2015}} and  label  the fixed point of the localization of path integral\cite{Localization}.

The Springer correspondence \cite{Collingwood:1993} is a injective map from the unipotent conjugacy classes of a simple group to the set of unitary representations of the Weyl group.   For the classical groups this map can be described explicitly in terms of partitions.  Unipotent conjugacy  of the (complexified) gauge group are classified by partitions. For Weyl group, irreducible unitary representations are in one-to one correspondence with ordered pairs of partitions $[\al;\bet]$.  $\al$ is a partition of $n_\al$ and $\bet$ is a partition of $n_{\beta}$ satisfying    $n_\al+n_\bet= n$.   Symbols introduced in \cite{Collingwood:1993} can be used to  describe the Springer correspondence for the classical groups, which  is a map related to the partitions on the two sides of the Springer correspondence.

Gukov and Witten  initiated to  study the $S$-duality for surface operators in $\cN=4$ super Yang-Mills theories  in \cite{Gukov:2006}\cite{Witten:2007}\cite{Gukov:2008}. A surface operator is defined by prescribing  a certain singularity structure of fields near the surface on which the operator is supported.  They are labelled by pairs of certain partitions.  The $S$-duality conjecture suggest that surface operators in  the theory with gauge group $G$ should have a counterpart  in the Langlands dual group $G^{L}$. A subclass of surface operators called {\it 'rigid'} surface operators   are expected to be closed under $S$-duality. Symbol is an invariant for the dual pair partition.  Using  symbol invariant, Wyllard made some explicit proposals for how the $S$-duality map should act on rigid surface operators in \cite{Wyllard:2009}. In \cite{Shou 1:2016}, we find a new subclass of rigid surface operators related by S-duality. We also simplify the construction of symbol invariant of partition. In \cite{Shou 2:2016}, we generalize the above method and propose a closed formula for the symbol of rigid partitions in the  $B_n, C_n$, and $D_n$ theories  uniformly.

In this paper, we attempt to extend the analysis in \cite{Gukov:2006}\cite{Wyllard:2009}\cite{Shou 1:2016} \cite{Shou 2:2016} to study the construction of symbol further. Firstly, we characterize the structure of symbol by two theorems\ref{L}\ref{ss}. We found a new method to construct the symbol, which is more elementary  and essential. This construction consist of two building blocks '\textbf{ Rule B }' and '\textbf{ Rule A}'.  Finally, using these building blocks, we proposed  closed formulas of symbol for partitions in  the $B_n, C_n$, and $D_n$ theories\ref{Ppb}\ref{Ppc}\ref{Ppd}.

The following is an outline of  this article.  In Section 2, we  introduce the concept of partition and some basic results related to the rigid partition. Then we introduce the definition of symbol  proposed in \cite{Shou 2:2016}.  We prove two   theorems  related to  the   structure  of  symbol.  In Section 3, we found  two elementary  building blocks in the construction of symbol. Then we extend these  two elementary  building blocks.  Using the structure theorems of symbol,  we can decompose the partition into several parts and calculate the contribution to symbol of each part independently.   Finally, according to these  decomposition   rules of partition, we give close formulas of symbols for rigid partitions in  the $B_n, C_n$ and $D_n$ theories. In Section 4, we make a comparison  of between the   closed formula in this paper and  the one in \cite{Shou 2:2016}.

With noncentral rigid conjugacy classes in $A_n$  and  more complicated exceptional groups,  we will concentrate on the $B_n, C_n$, and $D_n$ series in this paper. Other aspects of surface operators have been studied in \cite{Shou 1:2016}\cite{Gomis:2007}\cite{Drukker:2008}\cite{Gukov:2014}\cite{Sh06}.  The mismatch of the total number of rigid surface operators was found in the $B_n/C_n$ theories  through the study of the generating functions\cite{Henningson:2007a}\cite{Henningson:2007b}\cite{Wyllard:2008}\cite{Wyllard:2007}. There is another invariant of partition called fingerprint related to the Kazhdan-Lusztig map\cite{Spaltenstein:1992}\cite{Lusztig:1984}\cite{Symbol 2},  assumed to be  equivalent to the symbol invariant. Hopefully our constructions will be helpful in the proof of the equivalence. Clearly more work is required. Hopefully our constructions will be helpful in making further progress.

\section{Symbol  of  partition}
In this section, we introduce partition and relevant  results. We give   the definition of symbol and prove two structure theorems of symbol.
\subsection{Partitions in the  $B_n, C_n$, and $D_n$ theories}
Firstly, we will introduce the definition of symbol. They   are not identical  for  the  $B_n$, $C_n$, and $D_n$ theories. So we try to find equivalent  definitions which are convenient to study. We also introduce  formal  operations of  a  partition, which  are not only helpful to understand the new definitions of symbol but also helpful for the proof in  Section 3.


For the $B_n$($D_n$)theories, unipotent conjugacy  classes  are in one-to-one correspondence with partitions of $2n{+}1$($2n$) where all even integers appear an even number of times.  For the $C_n$ theories, unipotent conjugacy  classes  are in one-to-one correspondence with partitions $2n$ for which all odd integers appear an even number of times.  If it has no gaps (i.e.~$\la_i-\la_{i+1}\leq1$ for all $i$) and no odd (even) integer appears exactly twice, a partition in the $B_n$ or $D_n$ ($C_n$) theories is called {\it rigid}. We will focus on rigid partition in this paper.

The following propositions are pointed out in \cite{Wyllard:2009}, which are  used  to study the structure of symbol.   We give the  details of the  proof for the $B_n$ theory.
\begin{Pro}{\label{Pb}}
The longest row in a rigid  $B_n$ partition always contains an odd number of boxes. The following two rows of the first row are either both of odd length or both of even length.  This pairwise pattern then continues. If the Young tableau has an even number of rows the row of shortest length has to be even.
\end{Pro}
\begin{proof} For a partition in the $B_n$ theory,  even integers appear an even number of times.  So the sum of all odd integers is odd,    implying   the  number of odd integers is odd. Thus  the length of the partition is odd, which is the sum of the number of odd integers and  even integers. So the longest row in a rigid $B_n$ partition   contains an odd number of boxes.

If the following two rows of the first row are of different parities, then the difference of the length between these two rows is odd. It imply that part  '2' appears odd number of times  in the partition,  a  contradiction. So  the following two rows are either both of odd length or both of even length. Similarly, we can prove    next two rows are of  the same parities. This pairwise pattern continues.

The number of   boxes  of a pairwise pattern is even. If the Young tableau has an even number of rows, then the number of the total boxes of  the first row and the last row is old.  The shortest row is even,  since the longest row in a   $B_n$ partition always contains an odd number of boxes.
\end{proof}
\begin{flushleft}
\textbf{Remark:} If the last row is odd, the number of rows in the partition is odd.
\end{flushleft}

For partitions in the $C_n$ and $D_n$ theories, we have the following  propositions.
\begin{Pro}{\label{Pc}}
For a rigid $C_n$ partition, the longest two rows  both contain either  an even or an odd number  number of boxes.  This pairwise pattern then continues. If the Young tableau has an odd number of rows the row of shortest length has contain an even number of boxes.
\end{Pro}

\begin{Pro}{\label{Pd}}
For a rigid $D_n$ partition, the longest row  always contains an even number of boxes. And the following two rows are either both of even length or both of odd length. This pairwise pattern then continue. If the Young tableau has an even number of rows the row of the shortest length has to be even.
\end{Pro}

\subsection{Definition   and the  structure theorem  of  symbol}
It is convenient to use certain symbols to describe the Springer correspondence for the classical groups \cite{Collingwood:1993}.
The following  definitions of symbol is proposed in \cite{Shou 1:2016},  making the differences of definitions among  different theories   as small as possible.
\begin{Def}{\cite{Shou 1:2016}}\label{Dn}
\begin{itemize}
  \item $B_n$:   add $l-k$ to the $k$th part of the partition.
Arrange the odd parts of the sequence $l-k+\lambda_k$ in an increasing sequence $2f_i+1$  and arrange the even  parts in an increasing sequence $2g_i$.
Next calculate the following  terms
 \begin{equation*}
   \al_i = f_i-i+1\quad\quad\quad  \bet_i = g_i-i+1.
 \end{equation*}
 Finally we  write the {\it symbol} as
\begin{equation}\label{Db}
  \left(\ba{@{}c@{}c@{}c@{}c@{}c@{}c@{}c@{}} \al_1 &&\al_2&&\al_3&& \cdots \\ &\bet_1 && \bet_2 && \cdots  & \ea \right).
\end{equation}

  \item $C_n$: if the length of the partition is odd,  append an extra 0 as the last part of the partition. Then compute the symbol as in the $B_n$  case.
  \item $D_n$:  append an extra 0 as the last part of the partition, then compute the symbol as in the $B_n$  case.
\end{itemize}
\end{Def}

\begin{flushleft}
  \textbf{Remarks:}
  \begin{enumerate}
    \item  $l$ is the length of the partition  including the extra  0 appended   if needed  for the computation of the symbol.  The first step of the computation of  symbol in the $B_n$, $C_n$, and $D_n$ theories.
\begin{center}
\begin{tabular}{|r||c|c|c|c|}\hline
        $\lambda_{k}:$           & $\lambda_1$ & $\lambda_{2}$ &$\cdots$ & $\lambda_l$     \\ \hline
 $ l-k: $         & $l-1$ &$ l-2 $&$\cdots$ &$ 0 $                                         \\   \hline
 $l-k+\lambda_{k}: $& $l-1+\lambda_1$ &$ l-2+\lambda_2$&$\cdots $& $\lambda_l$              \\  \hline
\end{tabular}
\end{center}
  This table will be used frequently. The  terms   $f_i$  and   $g_i$   are calculated as follows
 \begin{center}
\begin{tabular}{|r||c|c|c|c|}\hline
  $2f_i+1: $      & $\cdots$ & $2f_2+1$ & $2f_1+1$   \\ \hline
    $ f_i:  $     & $\cdots $ & $f_2 $   & $f_1 $    \\ \hline
\end{tabular}
\qquad
\begin{tabular}{|r||c|c|c|}\hline
 $ 2g_i:$     & $\cdots $ & $2g_2$  &  $2g_1$ \\ \hline
   $ g_i:$      & $\cdots$ & $g_2$   &  $g_1$ \\ \hline
\end{tabular}
\end{center}
The sequences  $f_i$ and   $g_i$  increase from  right to left, corresponding to  sequences $\alpha_i$ and $\beta_i$,  respectively.
\item Since each term in the  sequence  $l-k+\lambda_{k}$  are independent, we can discuss the contribution to symbol  of each term $l-k+\lambda_{k}$ independently.
  \end{enumerate}
\end{flushleft}


Note that the form of the symbol (\ref{Db}) is a convenient notation. The relative positions between $\alpha_*$ and $\beta_*$  is not essential. We  prove the following structure theorem of symbol.
\begin{Pro}[Structure of the  Symbol]{\label{L}}  Let the  $k$th and $k+1$th rows are two rows of a pairwise pattern of the  partition $\lambda=m^{n_m}{(m-1)}^{n_{m-1}}\cdots{1}^{n_1}$.   If the length of the  $k$th row is even, the contribution to symbol of the parts $m^{n_m}{(m-1)}^{n_{m-1}}\cdots{k}^{n_k}$  is
\begin{equation}\label{Le}
  \Bigg(\!\!\!\ba{c}0\;\;0\cdots 0 \;\;  \alpha_{\mu-\frac{1}{2}L(k)+1}  \cdots \; \alpha_\mu  \\
\;\;\;0\cdots 0 \;\;\beta_{\nu-\frac{1}{2}L(k)+1} \cdots \; \beta_{\nu}\ \ea \Bigg)
\end{equation}
with $L(k)=\sum^{m}_{i=k} n_{i}$.

Let  the length of the $k$th row is old and is the second row  of a pairwise pattern. $l$ is the length of the partition  including the extra  0 appended   if needed  for the computation of the symbol. If $l-L(k)+\lambda_{L(k)}$ is odd, the contribution to symbol of the parts $m^{n_m}{(m-1)}^{n_{m-1}}\cdots{k}^{n_k}$  is
\begin{equation}\label{Lot}
  \Bigg(\!\!\!\ba{c}0\;\;0\;\cdots\; \alpha_{\mu-\frac{L(k)-1}{2}} \;\;  \alpha_{\mu-\frac{L(k)-1}{2}+1} \; \cdots \; \alpha_\mu  \\
\;\;\;0\;\;\cdots\;\; \;0 \;\;\beta_{\nu-\frac{L(k)-1}{2}+1} \; \cdots \; \beta_{\nu}\ \ea \Bigg)
\end{equation}
and the contribution to symbol of the parts $m^{n_m}{(m-1)}^{n_{m-1}}\cdots{{(k-1)}}^{n_{k-1}}$  is
\begin{equation}\label{Lot2}
  \Bigg(\!\!\!\ba{c}0\;\;0\;\cdots\; \alpha_{\mu-\frac{L(k-1)-1}{2}} \;\;  \alpha_{\mu-\frac{L(k-1)-1}{2}+1} \; \cdots \; \alpha_\mu  \\
\;\;\;0\;\;\cdots\;\; \;0 \;\;\beta_{\nu-\frac{L(k-1)-1}{2}+1} \; \cdots \; \beta_{\nu}\ \ea \Bigg).
\end{equation}
If  $l-L(k)+\lambda_{L(k)}$ is  even, the contribution to symbol of the parts $m^{n_m}{(m-1)}^{n_{m-1}}\cdots{k}^{n_k}$  is
\begin{equation}\label{Lob}
  \Bigg(\!\!\!\ba{c}0\;\;\;\;0\;\;\cdots \;\;0 \;\;  \alpha_{\mu-\frac{L(k)-1}{2}+1} \;\; \cdots \;\;\; \alpha_\mu  \\
\;\;\;0\;\cdots \;\beta_{\nu-\frac{L(k)-1}{2}} \;\;\beta_{\nu-\frac{L(k)-1}{2}+1} \;\;\cdots \;\; \beta_{\nu}\;\; \ea \Bigg).
\end{equation}
and the contribution to symbol of the parts $m^{n_m}{(m-1)}^{n_{m-1}}\cdots{(k-1)}^{n_{k-1}}$  is
\begin{equation}\label{Lob2}
  \Bigg(\!\!\!\ba{c}0\;\;\;\;0\;\;\cdots \;\;0 \;\;  \alpha_{\mu-\frac{L(k-1)-1}{2}+1} \;\; \cdots \;\;\; \alpha_\mu  \\
\;\;\;0\;\cdots \;\beta_{\nu-\frac{L(k-1)-1}{2}} \;\;\beta_{\nu-\frac{L(k-1)-1}{2}+1} \;\;\cdots \;\; \beta_{\nu}\;\; \ea \Bigg).
\end{equation}
\end{Pro}
\begin{proof} For the parts $\lambda_i, \lambda_{i+1}(\lambda_i=\lambda_{i+1})$, one of the  terms $l-i+\lambda_i, l-i+\lambda_{i+1} $ is odd  corresponding to $\alpha_*$ and the other  one is even corresponding to $\beta_*$.

If the last row  is not in a pairwise  pattern, then it is  even according to Proposition \ref{Pb}, \ref{Pc}, and \ref{Pd}.  For the first two parts $l-1+\lambda_1, l-2+\lambda_{2} (\lambda_1=\lambda_2=m)$, the odd term   is the biggest odd one  of and the even term  is the biggest even one of the sequence $l-i+\lambda_i$,   corresponding to $\alpha_\mu$ and  $\beta_\nu$, respectively. So  the contribution to symbol of the parts $(\lambda_1,\lambda_2)$ is
\begin{equation*}
  \Bigg(\!\!\!\ba{c}0\;\;0\cdots 0 \;\;  0  \cdots \;0\; \alpha_\mu  \\
\;\;\;0\cdots 0 \;\;0 \cdots \;0\; \beta_{\nu}\ \ea \Bigg).
\end{equation*}
This pairwise pattern of contribution to symbol will continue for the  parts  $m^{n_m}$. And  the contribution to symbol of the   parts $m^{n_m}$ is
 \begin{equation}\label{nm1}
  \Bigg(\!\!\!\ba{c}0\;\;0\cdots 0 \;\;  \alpha_{\mu-n_m/2+1}  \cdots \; \alpha_\mu  \\
\;\;\;0\cdots 0 \;\;\beta_{\nu-n_m/2+1} \cdots \; \beta_{\nu}\ \ea \Bigg)
\end{equation}
where the number of $\alpha_*$ is equal to that of  $\beta_*$ in the above formula.

\begin{figure}[!ht]
  \begin{center}
    \includegraphics[width=4.5in]{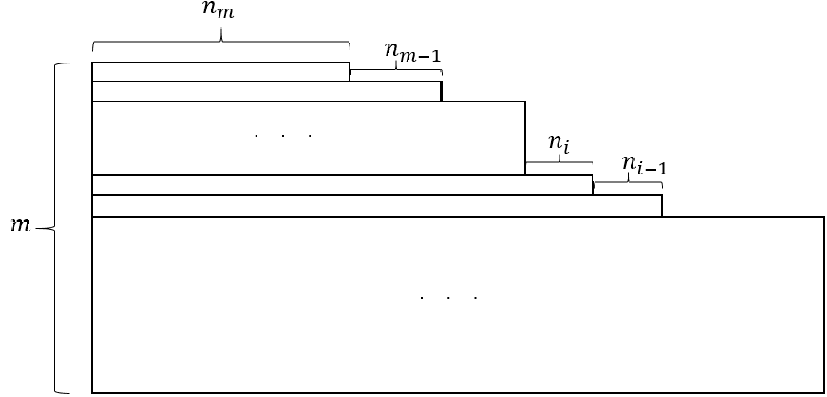}
  \end{center}
  \caption{The last two  rows of the Young tableau are in a pairwise  pattern. $n_m$ and $n_{m-1}$ are even.  Both the length of the  $i$th  and $(i-1)$th rows are odd.  $n_i$ is old  and $n_{i-1}$ is even.}
  \label{Stru}
\end{figure}
According to  remark 3 of Definition \ref{Dn}, the contributions to symbol of the  parts $m^{n_m}$ is independent of  the   parts ${(m-1)}^{n_{m-1}}\cdots{k}^{n_k}$ of the  partition $\lambda$. Without loss of generality,  we assume the last two rows are in a  pairwise patten. If these two rows  are even,  then both $n_{m}$ and  $n_{m-1}$ are even as shown in Fig.\ref{Stru}. By  repeating the procedure reaching the formula (\ref{nm1}),  the contribution to symbol   of the  parts $m^{n_m}$  is given by the formula (\ref{nm1}). In the same way, the contributions to symbol of  the parts $\lambda_{n_1+1}\cdots \lambda_{n_1+n_2}=(m-1)^{n_{m-1}}$ is
\begin{equation}\label{nm2}
  \Bigg(\!\!\!\ba{c}0\;\;0\cdots 0 \;\;  \alpha_{\mu-n_m/2-(n_{m-1})/2 +1} \cdots  \alpha_{\mu-n_m/2} \;\;\overbrace{ 0 \cdots  0 }^{n_m/2}\\
\;\;\;0\cdots 0 \;\; \beta_{\nu-n_m/2-(n_{m-1})/2 +1} \cdots  \beta_{\nu-n_m/2} \;\underbrace{0 \cdots  0}_{n_m/2} \ \ea \Bigg).
\end{equation}
Combing the  formulas (\ref{nm2}) and  (\ref{nm1}), we draw the conclusion of the  formula (\ref{Le}) for the parts $m^{n_m}{(m-1)}^{n_{m-1}}$.

The pattern of the formula (\ref{nm2}) will continue until an odd number $n_i$ which means both the lengths of the  $i$th  and $(i-1)$th rows are odd as shown in Fig.\ref{Stru}.
Since the contribution to symbol  of the parts  ${m}^{n_{m}}\cdots{(i-1)}^{n_{i-1}}$ is independent of the rest parts of the partition $\lambda$, without loss of generality, we assume the length of the last two rows are odd,  which means  $n_m$ is odd  and  $n_{m-1}$ is even as shown in Fig.\ref{Ev}.  We make a decomposition $m^{n_m}=m^{n_m-1}+m$. For the  parts $\lambda_1\cdots\lambda_{n_m-1}=m^{n_m-1}$,  repeating the procedure  reaching the formula (\ref{nm1}),  its contribution to symbol is
  \begin{equation}\label{nm1o}
  \Bigg(\!\!\!\ba{c}0\;\;0\cdots 0 \;\;  \alpha_{\mu-(n_m-1)/2+1}  \cdots \; \alpha_\mu  \\
\;\;\;0\cdots 0 \;\;\beta_{\nu-(n_m-1)/2+1} \cdots \; \beta_{\nu}\ \ea \Bigg).
\end{equation}

\begin{figure}[!ht]
  \begin{center}
    \includegraphics[width=4.5in]{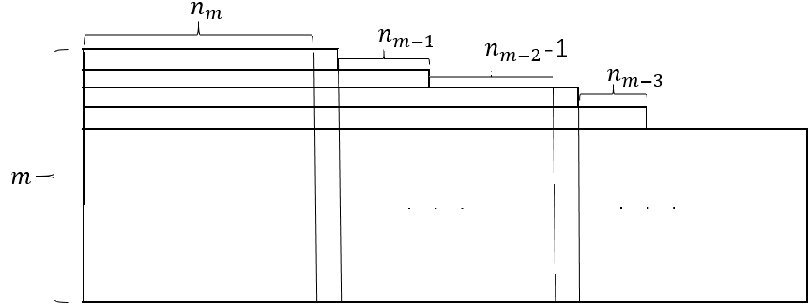}
  \end{center}
  \caption{The last two  rows  are old and in a pairwise  pattern. $n_m$ is old  and $n_{m-1}$ is even. Both the  ${m-2}$th  and $n_{m-3}$th rows are old.   $n_{m-2}$ is old  and $n_{m-3}$ is even.}
  \label{Ev}
\end{figure}
\begin{itemize}
  \item If $l-n_m+ \lambda_{n_m}$ is odd,   the  contribution to symbol of part $l-n_m+ \lambda_{n_m}$ is
\begin{equation}\label{nm11o}
\Bigg(\!\!\!\ba{c}0\;\;0\cdots 0 \;\;  0 \cdots \alpha_{\mu-(n_m-1)/2}\;\;0 \cdots  0 \\
\;\;\;0\cdots 0 \;\;0 \cdots \quad\quad 0\quad\quad \;0 \cdots  0 \ \ea \Bigg),
\end{equation}
which is the  formula (\ref{Lot}).
For the  parts $\lambda_{n_m+1}\cdots \lambda_{n_m+n_{m-1}}=(m-1)^{n_{m-1}}$, according to  the formula (\ref{nm2}), its contribution to symbol is
\begin{equation}\label{nm1m2o}
  \Bigg(\!\!\!\ba{c}0\;\;0\cdots 0 \;\;  \alpha_{\mu-(n_m-1)/2-n_{m-1}/2} \cdots  \alpha_{\mu-(n_m-1)/2-1} \;\; \overbrace{0 \cdots  0}^{(n_m+1)/2} \\
\;\;\;0\cdots 0 \;\; \beta_{\mu-(n_m-1)/2-n_{m-1}/2+1} \cdots  \beta_{\mu-(n_m-1)/2} \;\underbrace{0 \cdots  0}_{(n_m-1)/2} \ \ea \Bigg).
\end{equation}
Combing the formulas (\ref{nm1o}),  (\ref{nm11o}), and (\ref{nm1m2o}), we draw the conclusion of the  formula (\ref{Lot2}).

The pattern  of the formula (\ref{nm1m2o}) will continue until an odd $n_i$, which means both the $i$th  and $i-1$th rows are even. Without loss of generality, we assume both the length of the $(m-2)$th and $(m-3)$th  rows are even.  Then  $n_{m-2}$ is odd  and $n_{m-3}$ is even as shown in Fig.\ref{Ev}.
We make a decomposition as follows
$$\lambda_{n_m+n_{m-1}+1}\cdots\lambda_{n_m+n_{m-1}+n_{m-2}}=\lambda_{n_m+n_{m-1}+1}+\lambda_{n_m+n_{m-1}+2}\cdots\lambda_{n_m+n_{m-1}+n_{m-2}},$$
which is equal to $(m-2)^{n_{m-2}}=(m-2)+(m-2)^{n_{m-2}-1}$. Since   $l-n_m+ \lambda_{n_m}$ is odd and $n_{m-1}$ is even,
the following term $$l-(n_m+n_{m-1}+1)+ \lambda_{n_m+n_{m-1}+1}=l-n_m+ (\lambda_{n_m}-2)-(n_{m-1}+1)$$
is even.  The contribution  to symbol of the part $\lambda_{n_m+n_{m-1}+1}$ is
\begin{equation}\label{nm11oe}
\Bigg(\!\!\!\ba{c}0\;\;0\cdots 0 \;\;  0 \cdots  \quad\quad\quad 0\quad\quad \quad \;\;0 \cdots  0 \\
\;\;\;0\cdots 0 \;\;0 \cdots \beta_{\nu-(n_m+n_{m-1}+1)/2}  \;0 \cdots  0 \ \ea \Bigg).
\end{equation}
By combing the formulas (\ref{nm1o}), (\ref{nm1m2o}), (\ref{nm11o}),  and (\ref{nm11oe}),  the contribution to symbol of the parts $\lambda_1\cdots\lambda_{n_m+n_{m-1}}\lambda_{n_m+n_{m-1}+1}$ is
 \begin{equation}\label{nm1loall}
  \Bigg(\!\!\!\ba{c}0\;\;0\cdots 0 \;\;  \alpha_{\mu-(n_m+n_{m-1}+1)/2}  \cdots \; \alpha_\mu  \\
\;\;\;0\cdots 0 \;\;\beta_{\nu-(n_m+n_{m-1}+1)/2} \cdots \; \beta_{\nu}\ \ea \Bigg),
\end{equation}
which is the formula (\ref{Le}).
For the parts $\lambda_{n_m+n_{m-1}+2}\cdots \lambda_{n_m+n_{m-1}+n_{m-2}}=(m-2)^{n_{m-2}-1}$, repeating the procedure  reaching  the formula (\ref{nm1}), its contribution to symbol is
\begin{equation}\label{nm1m2oo}
  \Bigg(\!\!\!\ba{c}0\;\;0\cdots 0 \;\;  \alpha_{\mu-(n_m+n_{m-1}+n_{m-2})/2} \cdots  \alpha_{\mu-(n_m+n_{m-1}+1)/2} \;\; \overbrace{0 \cdots  0}^{(n_m+n_{m-1}+1)/2} \\
\;\;\;0\cdots 0 \;\; \beta_{\mu-(n_m+n_{m-1}+n_{m-2})/2} \cdots  \beta_{\nu-(n_m+n_{m-1}+1)/2} \;\underbrace{0 \cdots  0}_{(n_m+n_{m-1}+1)/2} \ \ea \Bigg).
\end{equation}
Similarly, for the parts $\lambda_{n_m+n_{m-1}+n_{m-2}+1}\cdots \lambda_{n_m+n_{m-1}+n_{m-2}+n_{m-3}}=(m-3)^{n_{m-3}}$,  its contribution to symbol is
\begin{equation}\label{nm1m2o13}
  \Bigg(\!\!\!\ba{c}0\;\;0\cdots 0 \;\;  \alpha_{\mu-(n_m+n_{m-1}+n_{m-2}+n_{m-3})/2} \cdots  \alpha_{\mu-(n_m+n_{m-1}+n_{m-2}+1)/2} \;\; \overbrace{0 \cdots  0}^{(n_m+n_{m-1}+n_{m-2})/2} \\
\;\;\;0\cdots 0 \;\; \beta_{\mu-(n_m+n_{m-1}+n_{m-2}+n_{m-3})/2} \cdots  \beta_{\nu-(n_m+n_{m-1}+n_{m-2}+1)/2} \;\underbrace{0 \cdots  0}_{(n_m+n_{m-1}+n_{m-2})/2} \ \ea \Bigg).
\end{equation}
Combing the  formulas (\ref{nm1loall}), (\ref{nm1m2oo}), and (\ref{nm1m2o13}), we get the formula (\ref{Le}). The pattern  of formula (\ref{nm1m2o13}) will continue until an odd $n_i$,  which means both the length of the  $i$th  and $(i-1)$th rows are old. This return back at the beginning of the proof.
  \item
If  $l-n_m+ \lambda_{n_m}$ is even,  the  contribution  to symbol of the  part $\lambda_{n_m}=m$ is
\begin{equation}\label{nm11e}
\Bigg(\!\!\!\ba{c}0\;\;0\cdots 0 \;\;  0 \cdots \quad\quad  0\quad\quad  \;\;0 \cdots  0 \\
\;\;\;0\cdots 0 \;\;0 \cdots  \beta_{\nu-(n_m-1)/2}\;0 \cdots  0 \ \ea \Bigg).
\end{equation}
For the parts $\lambda_{n_m+1}\cdots \lambda_{n_m+n_{m-1}}=(m-1)^{n_{m}-1}$,  its contribution to symbol is
\begin{equation}\label{nm1m2e}
   \Bigg(\!\!\!\ba{c}0\;\;0\cdots 0 \;\;  \alpha_{\mu-(n_m-1)/2-n_{m-1}/2+1} \cdots  \alpha_{\mu-(n_m-1)/2} \;\; \overbrace{0 \cdots  0}^{(n_m-1)/2} \\
\;\;\;0\cdots 0 \;\; \beta_{\mu-(n_m-1)/2-n_{m-1}/2} \cdots  \beta_{\mu-(n_m-1)/2-1} \;\underbrace{0 \cdots  0}_{(n_m+1)/2} \ \ea \Bigg).
\end{equation}
We   draw the conclusion of formula (\ref{Lob}). The other process of the proof is   similar to the case that $l-n_m+ \lambda_{n_m}$ is old.
\end{itemize}

\end{proof}
\begin{flushleft}
  \textbf{Remark:} According to the proof,  the total number   of  $\alpha_*$ and $\beta_*$ is  $l$.
\end{flushleft}

By using the above theorem,  the   symbol  has following  concise form. Here we prove the theorem using the structure theorem Proposition \ref{L}.
\begin{Pro}{\cite{Shou 1:2016}}\label{ss}
The symbol of  the partition $\lambda=m^{n_m}{(m-1)}^{n_{m-1}}\cdots{1}^{n_1}$ in the  $B_n, C_n$, and $D_n$ theories has the following compact form
 \be\label{Dt}
    \boxed{\left(\ba{@{}c@{}c@{}c@{}c@{}c@{}c@{}c@{}c@{}c@{}} \alpha_1 &&  \alpha_2 &&\cdots  &&  \alpha_m  \\ & \beta_1 && \cdots && \beta_{m+t} && & \ea \right)}
\ee
where  $m=(l+(1-(-1)^l)/2)/2$ and  $l$ is the length of the partition.   $t=-1$ for the $B_n$ theory,  $t=0$  for the $C_n$ theory,  and  $t=1$ for the  $D_n$ theory.
\end{Pro}
\begin{proof} \begin{itemize}
                  \item Let $\lambda$ is a partition in the $B_n$ theory.
                  \begin{enumerate}
                  \item If both the second and third rows are old, $l-L(3)+\lambda_{L(3)}$ is old.   According to the formula (\ref{Lot2}), the parts $\lambda_1\cdots\lambda_{L(2)}$ contribute one more  $\alpha_*$ than  $\beta_*$.  Since the first row is odd,  $n_1$ is even. The parts $\lambda_{L(2)+1}\cdots\lambda_{L(1)}=1^{n_1}$ contribute equal  number of   $\alpha_*$ and $\beta_*$.
                      \item If both the second and third rows are even, according to the  formula (\ref{Le}), the parts $\lambda_1\cdots\lambda_{L(2)}$ contribute equal number of    $\alpha_*$ and  $\beta_*$.  Since the first row is odd,  $n_1$ is odd.  We can decompose the parts $\lambda_{L(2)+1}\cdots\lambda_{L(1)}=1^{n_1}$ into two parts $\lambda_{L(2)+1}\cdots\lambda_{L(1)-1}=1^{n_1-1}$ and $\lambda_{L(1)}=1$. The former one  contribute equal  number of    $\alpha_*$ and $\beta_*$.  The latter one corresponds to  $\alpha_1$.
                          \end{enumerate}

                  \item Let $\lambda$ is a partition in the $C_n$ theory.
                  \begin{enumerate}
                                            \item If both the first two rows are odd, we append an extra 0 as the last part of the partition. The part $\lambda_{L(1)+1}=0$ corresponds to  $\beta_1=0$. $l-L(2)+\lambda_{L(2)}$  is old.  According to formula (\ref{Lot2}), the parts $\lambda_1\cdots\lambda_{L(2)}$ contribute equal number of   $\alpha_*$ and  $\beta_*$. Since the first row is even,  $n_1$ is even. The parts $\lambda_{L(2)+1}\cdots\lambda_{L(1)}=1^{n_1}$ contribute equal  number of   $\alpha_*$ and $\beta_*$. Since the  part $\lambda_{L(1)+1}=0$ corresponds to $\beta_1=0$,  there are equal number of $\alpha_*$ and $\beta_*$.
                                            \item If both the first two  rows are even, according to  the formula (\ref{Le}), the parts $\lambda_1\cdots\lambda_{L(1)}$ contribute equal number of   $\alpha_*$ and  $\beta_*$.
                                          \end{enumerate}
                  \item Let $\lambda$ is a partition in the $D_n$ theory. We append an extra 0 as the last part of partition. The part $\lambda_{L(1)+1}=0$ corresponds to $\beta_1=0$.
                    \begin{enumerate}
                  \item If both the second and third rows are old, $l-L(3)+\lambda_{L(3)}$ is old.   According to the formula (\ref{Lot2}), the parts $\lambda_1\cdots\lambda_{L(2)}$ contribute one more  $\alpha_*$ than  $\beta_*$.  Since the first row is even,  $n_1$ is old. We can decompose the parts $\lambda_{L(2)+1}\cdots\lambda_{L(1)}=1^{n_1}$ into two parts $\lambda_{L(2)+1}\cdots\lambda_{L(1)-1}=1^{n_1-1}$ and $\lambda_{L(1)}=1$. The former one  contribute equal  number of    $\alpha_*$ and $\beta_*$ and the latter one corresponds to  $\beta_*$. Since the  part $\lambda_{L(1)+1}=0$ corresponds to $\beta_1=0$,   the total number of  $\beta_*$  is one more than that of $\alpha_*$.
                  \item If both the second and third rows are even, according to the  formula (\ref{Le}), the parts $\lambda_1\cdots\lambda_{L(2)}$ contribute equal number of    $\alpha_*$ and  $\beta_*$.  Since the first row is even,  $n_1$ is even. The parts $\lambda_{L(2)+1}\cdots\lambda_{L(1)}=1^{n_1}$ contribute equal  number of   $\alpha_*$ and $\beta_*$. Since the  part $\lambda_{L(1)+1}=0$ corresponds to $\beta_1=0$,   the total number of  $\beta_*$  is one more than that of $\alpha_*$.
                          \end{enumerate}

                \end{itemize}
\end{proof}

\section{Construction of  symbol }
Firstly, we  introduce two formal operations of a partition,  described   as \textbf{ 'Rule A' } and \textbf{ 'Rule B'}. These formal  operations  can be used to derive  the closed formulas of symbols in the  $B_n$, $C_n$, and $D_n$  theories. Let $l$ be the length of the partition  including the extra  0 appended   according to Definition \ref{Dn}.

\subsection{Building blocks of  symbol }
 The contribution to  symbol of the part $\lambda_i$ can be seen as  the sum of the contribution of each box of it formally.   We decompose  $\lambda_i$ into several  parts whose    contribution to symbol are easy to compute.

Adding a column $2^1$ to the part $\lambda_i$,   we  discuss what happen to the \textit{corresponding} entry in the  symbol.
\begin{itemize}
\item If $l-i+\lambda_i$ is even,   we have
$$l-i+\lambda_i+2=2g_{i_a}+2=2(g_{i_a}+1)=2g^{'}_{i_a}.$$
Let $g^{'}_{i_a}$ be the notation $g_i$ for the part $\lambda_i+2^1$, which means $\beta^{'}_{i_a}=\beta_{i_a}+1$.
\item  If $l-i+\lambda_i$ is  odd,   we have
$$l-i+\beta_i+2=2f_{i_b}+1+1=2(f_{i_b}+1)+1=2f^{'}_{i_b}$$
which means $\alpha^{'}_{i_b}=\alpha_{i_b}+1$.
\end{itemize}
We can summary the above results as the following rule
\begin{flushleft}
 \textbf{ Rule A: } After  adding two boxes to  the  part $\lambda_i$,   the  \textit{corresponding} entry  of  the symbol  increase by  one.
\end{flushleft}

Next,  adding a row $1^2$ to  parts $\lambda_i,\lambda_{i+1}(\lambda_i=\lambda_{i+1})$,  we  discuss what happen to the corresponding entries of the  symbol.
\begin{itemize}
\item If  $l-i+\lambda_i$ is even, we have
\begin{equation*}
l-i+\lambda_i=2g_{i_a}, \quad\quad  l-(i+1)+\lambda_{i+1}=2f_{i_b}+1.
\end{equation*}
The first term corresponds to $\beta_{i_a}$  and  the second one  corresponds to $\alpha_{i_b}$.    Since $\lambda_i=\lambda_{i+1}$, we have
\begin{equation*}
  l-(i+1)+\lambda_{i+1}=2f_{i_b}+1=2(g_{i_a}-1)+1
\end{equation*}
which means $f_{i_b}=g_{i_a}-1$.
After  adding a row $1^2$ to  parts $\lambda_i,\lambda_{i+1}(\lambda_i=\lambda_{i+1})$, we have
\begin{equation*}
l-i+\lambda_i+1=2g_{i_a}+1, \quad\quad  l-(i+1)+\lambda_{i+1}+1=2(f_{i_b}+1).
\end{equation*}
Substituting  $f_{i_b}=g_{i_a}-1$ into the above identities,  we obtain
\begin{equation*}
 l-i+\lambda_i+1=2(f_{i_b}+1)+1, \quad\quad  l-(i+1)+\lambda_{i+1}+1=2g_{i_a}
\end{equation*}
which means $\beta^{'}_{i_a}=\beta_{i_a}$ and $\alpha^{'}_{i_b}=\alpha_{i_b}+1$.
\item  If $l-i+\lambda_i$ is  odd, similarly,   we have
$$\beta^{'}_{i_a}=\beta_{i_a}+1, \quad\quad \alpha^{'}_{i_b}=\alpha_{i_b}.$$
\end{itemize}
We can summary   the above results as the following rule.
\begin{flushleft}
 \textbf{ Rule B: } After adding  a row $1^2$ to the   parts $\lambda_i,\lambda_{i+1}(\lambda_i=\lambda_{i+1})$ formally, the entry of the symbol corresponding to  $\lambda_{i+1}$  increase by  one  while  the entry corresponding to  $\lambda_{i}$ do not change.
\end{flushleft}

 '\textbf{Rule A}' and '\textbf{Rule B}' are of prime importance.  Using these rules, we prove formulas  which are the   building  blocks of the closed formula of the symbol.
 After adding a row $1^{2m}$  to the parts  $\lambda_{i} \cdots \lambda_{i+2m-1}(\lambda_{i}=\cdots=\lambda_{i+2m-1})$,  we  discuss what happen to the  corresponding entries in the  symbol.
\begin{itemize}
  \item If $l-i+\lambda_i$ is odd, the sequence
$$( l-i+\lambda_i, l-i-2+\lambda_{i+2}, \cdots,  l-i-2m+2 + \lambda_{i+2m-2})$$
correspond to entries  $(\alpha_{i_a} \cdots \alpha_{i_{a+m-1}})$ from right to left  in the top row of the  symbol while
$$(l-i-1+\lambda_{i+1}, l-i-3+\lambda_{i+3}, \cdots,  l-i-2m+1 + \lambda_{i+2m-1})$$
corresponds to $(\beta_{i_a} \cdots \beta_{i_{a+m-1}})$  from right to left in the bottom  row of the symbol.  By using '\textbf{Rule B}' $m$ times,   these $2m$ boxes   contribute  to symbol as follows
\begin{equation}\label{brb}
\Bigg(\!\!\!\ba{c}0\;\;0\cdots 0 \;\; \overbrace{ 0 \cdots  0}^{m} \;\;0 \cdots  0 \\
\;\;\;0\cdots 0 \;\;\underbrace{ 1 \cdots  1}_{m}\;0 \cdots  0 \ \ea \Bigg)
\end{equation}
  \item If $l-i+\lambda_i$ is even, the sequence
$$( l-i+\lambda_i, l-i-2+\lambda_{i+2}, \cdots,  l-i-2m+2 + \lambda_{i+2m-2})$$
correspond to entries  $(\beta_{i_a} \cdots \beta_{i_{a+m-1}})$ from right to left in the bottom row of symbol while
$$(l-i-1+\lambda_{i+1}, l-i-3+\lambda_{i+3}, \cdots,  l-i-2m+1 + \lambda_{i+2m-1})$$
correspond to  $(\alpha_{i_a} \cdots \alpha_{i_{a+m-1}})$ in the top row of symbol.
After adding a row with $1^{2m}$ boxes, using '\textbf{Rule B}' $m$ times,  we  find that these $2m$ boxes   contribute  to symbol as follows
\begin{equation}\label{brt}
  \Bigg(\!\!\!\ba{c}0\;\;0\cdots 0 \;\; \overbrace{ 1 \cdots  1}^{m} \;\; 0 \cdots  0 \\
\;\;\;0\cdots 0 \;\;\underbrace{ 0 \cdots  0}_{m}\;0 \cdots  0 \ \ea \Bigg)
\end{equation}
\end{itemize}

The formulas (\ref{brb}) and (\ref{brt}) are two extended  versions of '\textbf{Rule B}'. Next  we introduce another two  formulas combing  '\textbf{Rule A}' and '\textbf{Rule B}'.  For  the    parts  $\lambda_{i}, \cdots, \lambda_{i+2m}(\lambda_{i}=\cdots=\lambda_{i+2m})$, we discuss what happen to  the symbol    after adding a partition  $21^{2m}$.    We  add the column $2$ to $\lambda_i$ firstly, and then add the row $1^{2m}$ to the  parts $\lambda_{i+1}, \cdots, \lambda_{i+2m-1},\lambda_{i+2m}$.
\begin{itemize}
  \item  If  $l-i+\lambda_i$ is odd, it must  corresponds to a term $\alpha_{i_b}$. After adding the column $2$, using '\textbf{Rule A}', we have $\alpha^{'}_{i_b}=\alpha_{i_b}+1$.

  The sequence
$$(  l-i-2+\lambda_{i+2}, \cdots,  l-i-2m + \lambda_{i+2m})$$
corresponds to entries  $(\alpha_{i_{b+1}}, \cdots, \alpha_{i_{b+m}})$ from right to left  in the top row of symbol while
$$(l-i-1+\lambda_i, l-i-3+\lambda_{i+3}, \cdots,  l-i-2m+1 + \lambda_{i+2m-1})$$
corresponds to $(\beta_{i_a} \cdots \beta_{i_{a+m-1}})$ from right to left in the bottom row of symbol. After  adding  the row $1^{2m}$, using '\textbf{Rule B}' $m$ times, we have
$(\alpha_{i_{b+1}}+1,  \cdots,  \alpha_{i_{b+m}}+1)$.
Thus, the partition $21^{2m}$  contribute to symbol as follows
\begin{equation}\label{bzt}
 \Bigg(\!\!\!\ba{c}0\;\;0\cdots 0 \;\; \overbrace{ 1 \cdots  1}^{m+1} \;\; 0 \cdots  0 \\
\;\;\;0\cdots 0 \;\;\underbrace{ 0 \cdots  0}_{m+1}\;0 \cdots  0 \ \ea \Bigg).
\end{equation}

  \item If  $l-i+\lambda_i$ is even, it must   corresponds to a term $\beta_{i_a}$. After adding the column $2$, using '\textbf{Rule A}', we have $\beta^{'}_{i_a}=\beta_{i_a}+1$.

  The sequence
$$(  l-i-2+\lambda_{i+2}, \cdots,  l-i-2m + \lambda_{i+2m})$$
corresponds to entries  $(\beta_{i_{a+1}}, \cdots, \beta_{i_{a+m}})$ from right to left in the bottom row of symbol while
$$(l-i-1+\lambda_i, l-i-3+\lambda_{i+3}, \cdots,  l-i-2m+1 + \lambda_{i+2m-1})$$
corresponds to $(\alpha_{i_b} \cdots \alpha_{i_{b+m-1}})$  in the top row of symbol. After adding  the row $1^{2m}$, using '\textbf{Rule B}' $m$ times, we have
$(\beta_{i_{a+1}}+1,  \cdots,  \beta_{i_{a+m}}+1)$.
Thus, the partition $21^{2m}$  make a contribution to symbol as follows
\begin{equation}\label{bzb}
\Bigg(\!\!\!\ba{c}0\;\;0\cdots 0 \;\; \overbrace{ 0 \cdots  0}^{m+1} \;\;0 \cdots  0 \\
\;\;\;0\cdots 0 \;\;\underbrace{ 1 \cdots  1}_{m+1}\;0 \cdots  0 \ \ea \Bigg).
\end{equation}
\end{itemize}

We can summary above results as follows
\begin{center}
\begin{tabular}{|c|c|c|c|}\hline
\multicolumn{4}{|c|}{ Contributions   of the  parts added to $\lambda_i\cdots\lambda_{i+2m}$ or $\lambda_i\cdots\lambda_{i+2m+1}$} \\ \hline
Type &  Parts  & Parity of $l-i+\lambda_i$ & Contribution to symbol    \\ \hline
eo&$1^{2m}$ & odd  & $\Bigg(\!\!\!\ba{c}0\;\;0\cdots 0 \;\; \overbrace{ 0 \cdots  0}^{m} \;\;0 \cdots  0 \\
\;\;\;0\cdots 0 \;\;\underbrace{ 1 \cdots  1}_{m}\;0 \cdots  0 \ \ea \Bigg)$  \\ \hline
ee&$1^{2m}$ & even  & $\Bigg(\!\!\!\ba{c}0\;\;0\cdots 0 \;\; \overbrace{ 1 \cdots  1}^{m} \;\; 0 \cdots  0 \\
\;\;\;0\cdots 0 \;\;\underbrace{ 0 \cdots  0}_{m}\;0 \cdots  0 \ \ea \Bigg)$    \\ \hline
oo&$1^{2m+1}$ & odd  & $ \Bigg(\!\!\!\ba{c}0\;\;0\cdots 0 \;\; \overbrace{ 1 \cdots  1}^{m+1} \;\; 0 \cdots  0 \\
\;\;\;0\cdots 0 \;\;\underbrace{ 0 \cdots  0}_{m+1}\;0 \cdots  0 \ \ea \Bigg)$   \\ \hline
oe&$1^{2m+1}$ & even  & $\Bigg(\!\!\!\ba{c}0\;\;0\cdots 0 \;\; \overbrace{ 0 \cdots  0}^{m+1} \;\;0 \cdots  0 \\
\;\;\;0\cdots 0 \;\;\underbrace{ 1 \cdots  1}_{m+1}\;0 \cdots  0 \ \ea \Bigg)$     \\ \hline
\end{tabular}
\end{center}

Now, we introduce another  building block of the  construction of the symbol. We discuss what happen to the  symbol  after adding a partition $2^2$ to the  parts $\lambda_i,\lambda_{i+1}(\lambda_i=\lambda_{i+1})$. We  add the column $2$ to $\lambda_i$, and then add the column $2$ to $\lambda_{i+1}$. If the parts $\lambda_i,\lambda_{i+1}$ corresponds to the entries $\alpha_{i_a},\beta_{i_b}$ in the symbol, parts $\lambda_i+2,\lambda_{i+1}+2$ correspond to the entries $\alpha_a+1,\beta_b+1$ by using '\textbf{Rule A}' twice. If the parts $\lambda_i,\lambda_{i+1}$ corresponds to the entries $\beta_{i_a},\alpha_{i_b}$ in the symbol,  parts $\lambda_i+2,\lambda_{i+1}+2$ correspond to the entries $\beta_{i_a}+1,\alpha_{i_b}+1$ by using '\textbf{Rule B}' twice. In any case,  the partition $2^2$  contribute to symbol as follows
\begin{equation}\label{b22}
\Bigg(\!\!\!\ba{c}0\;\;0\cdots 0 \;\;  0 \cdots  1 \;\;0 \cdots  0 \\
\;\;\;0\cdots 0 \;\; 0 \cdots  1\;0 \cdots  0 \ \ea \Bigg).
\end{equation}
In another way, we  add the row $1^2$ to $\lambda_i,\lambda_{i+1}$ firstly, and then add the row $1^2$. By using '\textbf{Rule B}' twice,  the contribution to symbol  of $2^2$ is  consistent with the formula (\ref{b22}).

Next we introduce an extended version of the  formula (\ref{b22}). We add a  partition $2^{2r}$  to parts $\lambda_i,\lambda_{i+1}(\lambda_i=\lambda_{i+1})$. If the  parts $\lambda_i,\lambda_{i+1}$ corresponds to none zero entries in the symbol as follows
\begin{equation*}
\Bigg(\!\!\!\ba{c}0\;\;0 \cdots 0 \;   \cdots  0 \; * \;\;0 \cdots  0 \\
\;\;\;0\cdots 0 \;  \cdots 0 \; *\;0 \cdots  0 \ \ea \Bigg)
\end{equation*}
By using the  formula $(\ref{b22})$ $r$ times, the contribution to symbol  of the partition $2^{2r}$  is
 \begin{equation}\label{b22r}
\Bigg(\!\!\!\ba{c}0\;\;0 \cdots  0 \;   \cdots \; 0  \; r \;0 \cdots  0 \\
\;\;\;0\cdots 0 \; \cdots 0 \; r\;0 \cdots  0 \ \ea \Bigg).
\end{equation}
We can generalize the above formula further. We add the partition $(2m)^{2r}$  to the parts $\lambda_{i} \cdots \lambda_{i+2m-1}(\lambda_{i}=\cdots=\lambda_{i+2m-1})$. If  the parts $\lambda_{i} \cdots \lambda_{i+2m-1}$ corresponds to none zero entries in the symbol as follows
\begin{equation*}
\Bigg(\!\!\!\ba{c}0\;\;0\cdots 0 \;\;  \overbrace{* \cdots  *}^m \;\;0 \cdots  0 \\
\;\;\;0\cdots 0 \;\; \underbrace{* \cdots  *}_m \;0 \cdots  0 \ \ea \Bigg),
\end{equation*}
by using the formula $(\ref{b22})$ $r$ times, the contribution to symbol of the partition $(2m)^{2r}$  is
 \begin{equation}\label{b22r}
\Bigg(\!\!\!\ba{c}0\;\;0\cdots 0 \;\; \overbrace{ r \cdots  r}^m \;\;0 \cdots  0 \\
\;\;\;0\cdots 0 \;\; \underbrace{r \cdots  r}_m\;0 \cdots  0 \ \ea \Bigg).
\end{equation}

\subsection{Symbol of partitions in the $B_n$ theory}
First, we illustrate our strategy for the  computation of   symbol by an example.
\begin{flushleft}
\textbf{Example:}  $B_{72}$,    $\lambda=9^48^27^36^45^44^23^42^21^4$, the Young tableau is
\be\label{Yb1}
\tableau{4 6 9 13 17 19 23 25 29}\non
\ee
\end{flushleft}
 Since the number of rows in this partition is old,   according to Proposition \ref{Pb}, the last two row have the same parity.  We can decompose the above Young tableaux  into two parts formally: the first row  and the rest parts of the partition
\be
 \tableau{29}'+'\tableau{4 6 9 13 17 19 23 25}\non
\ee
$'+'$ indicate it is a formal addition. The  symbol of $\lambda$  is the sum of  the contributions   of these two parts. The first row  can be seen as a partition in the $B_n$ theory,  whose contribution to  symbol can be  computed directly
\begin{equation}\label{B1}
\left(\ba{@{}c@{}c@{}c@{}c@{}c@{}c@{}c@{}c@{}c@{}c@{}c@{}c@{}c@{}c@{}c@{}c@{}c@{}c@{}c@{}c@{}c@{}c@{}c@{}c@{}c@{}c@{}c@{}c@{}c@{}} 0&&0&&0&&0&&0&&0&&0&&0&&0&&0&&0&&0&&0&&0&&0 \\ &1 && 1 && 1 &&1&&1&&1 && 1 &&1&&1&&1 && 1 &&1&&1&&1 & \ea \right).
\end{equation}
We decompose the second Young tableaux into blocks whose contributions to symbol can be computed by the  formulas in the previous  subsection. We have the following decomposition
\begin{eqnarray}\label{Gp}
 &&  \tableau{4 6 9 13 17 19 23 25}\nonumber\\
 &=& \tableau{4 4 4 4 4 4 4 4 4} \,{'+'}\, \tableau{2 4 4 4 4 4 4} \,{'+'}\,  \tableau{1 5 8 8 8 8} \,{'+'}\, \tableau{1 3 6 6} \,{'+'}\, \tableau{1 3} \\
  &=&YB_1\,\,{'+'}\,\,YB_2\,\,{'+'}\,\,YB_3\,\,{'+'}\,\,YB_4\, \,{'+'}\,\,YB_5. \nonumber
\end{eqnarray}
We denote the five Young tableaux  in the first equality as $YB_1, YB_2, YB_3, YB_4$, and $YB_5$, respectively. Each Young tableaux  $YB_i$ can be decomposed into two parts.
\begin{itemize}
  \item The Young tableaux $YB_1$ has a decomposition
\begin{equation*}
  \tableau{4 4 4 4 4 4 4 4}=  \tableau{4 4 4 4 4 4 4 4} + \emptyset.
\end{equation*}
By using the  formula (\ref{b22r}),  the first Young tableaux on the right side  contribute to symbol as follows
\begin{equation}\label{YB11}
\left(\ba{@{}c@{}c@{}c@{}c@{}c@{}c@{}c@{}c@{}c@{}c@{}c@{}c@{}c@{}c@{}c@{}c@{}c@{}c@{}c@{}c@{}c@{}c@{}c@{}c@{}c@{}c@{}c@{}c@{}c@{}} 0&&0&&0&&0&&0&&0&&0&&0&&0&&0&&0&&0&&0&&4&&4\\ &0 && 0 && 0 &&0&&0&&0&& 0 &&0&&0&&0&& 0 &&0&&4&&4 & \ea \right).
\end{equation}
  \item The Young tableaux $YB_2$ has a decomposition
\begin{equation*}
  \tableau{2 4 4 4 4 4 4}= \tableau{4 4 4 4 4 4} + \tableau{2}
\end{equation*}
By using the formula (\ref{b22r}),  the first Young tableaux on the right side  contribute to symbol as follows
\begin{equation}\label{BY31}
\left(\ba{@{}c@{}c@{}c@{}c@{}c@{}c@{}c@{}c@{}c@{}c@{}c@{}c@{}c@{}c@{}c@{}c@{}c@{}c@{}c@{}c@{}c@{}c@{}c@{}c@{}c@{}c@{}c@{}c@{}c@{}} 0&&0&&0&&0&&0&&0&&0&&0&&0&&0&&0&&3&&3&&0&&0 \\ &0 && 0 && 0 &&0&&0&&0&& 0 &&0&&0&&0&& 3 &&3&&0&&0 & \ea \right).
\end{equation}
By using the formula (\ref{brb}),  the second one  contribute to symbol as follows
\begin{equation}\label{BY32}
\left(\ba{@{}c@{}c@{}c@{}c@{}c@{}c@{}c@{}c@{}c@{}c@{}c@{}c@{}c@{}c@{}c@{}c@{}c@{}c@{}c@{}c@{}c@{}c@{}c@{}c@{}c@{}c@{}c@{}c@{}c@{}} 0&&0&&0&&0&&0&&0&&0&&0&&0&&0&&0&&0&&0&&0&&0 \\ &0 && 0 && 0 &&0&&0&&0&& 0 &&0&&0&&0&& 1 &&0&&0&&0 & \ea \right).
\end{equation}
  \item The Young tableaux $YB_3$ has a decomposition
\begin{equation*}
  \tableau{1 5 8 8 8 8}= \tableau{8 8 8 8}+ \tableau{1 5}
\end{equation*}
By using the (\ref{b22r}),  the first Young tableaux on the right side  contribute to symbol as follows
\begin{equation}\label{BY31}
\left(\ba{@{}c@{}c@{}c@{}c@{}c@{}c@{}c@{}c@{}c@{}c@{}c@{}c@{}c@{}c@{}c@{}c@{}c@{}c@{}c@{}c@{}c@{}c@{}c@{}c@{}c@{}c@{}c@{}c@{}c@{}} 0&&0&&0&&0&&0&&0&&0&&0&&0&&4&&4&&0&&0&&0&&0 \\ &0 && 0 && 0 &&0&&0&&0&& 0 &&0&&4&&4&& 0 &&0&&0&&0 & \ea \right).
\end{equation}
By using the (\ref{bzt}),  the second one  contribute to symbol as follows
\begin{equation}\label{BY32}
\left(\ba{@{}c@{}c@{}c@{}c@{}c@{}c@{}c@{}c@{}c@{}c@{}c@{}c@{}c@{}c@{}c@{}c@{}c@{}c@{}c@{}c@{}c@{}c@{}c@{}c@{}c@{}c@{}c@{}c@{}c@{}} 0&&0&&0&&0&&0&&0&&0&&0&&1&&1&&1&&0&&0&&0&&0 \\ &0 && 0 && 0 &&0&&0&&0&& 0 &&0&&0&&0&& 0 &&0&&0&&0 & \ea \right).
\end{equation}
  \item The Young tableaux $YB_4$ has a decomposition
\begin{equation*}
  \tableau{1 3 6 6}= \tableau{6 6} + \tableau{1 3}
\end{equation*}
By using the formula (\ref{b22r}),  the first Young tableaux on the right side  contribute to symbol as follows
\begin{equation}\label{BY31}
\left(\ba{@{}c@{}c@{}c@{}c@{}c@{}c@{}c@{}c@{}c@{}c@{}c@{}c@{}c@{}c@{}c@{}c@{}c@{}c@{}c@{}c@{}c@{}c@{}c@{}c@{}c@{}c@{}c@{}c@{}c@{}} 0&&0&&0&&0&&1&&1&&1&&0&&0&&0&&0&&0&&0&&0&&0 \\ &0 && 0 && 0 &&1&&1&&1&& 0 &&0&&0&&0&& 0&&0&&0&&0 & \ea \right).
\end{equation}
By using the formula  (\ref{bzt}),  the second one  contribute to symbol as follows
\begin{equation}\label{BY32}
\left(\ba{@{}c@{}c@{}c@{}c@{}c@{}c@{}c@{}c@{}c@{}c@{}c@{}c@{}c@{}c@{}c@{}c@{}c@{}c@{}c@{}c@{}c@{}c@{}c@{}c@{}c@{}c@{}c@{}c@{}c@{}} 0&&0&&0&&0&&0&&1&&1&&0&&0&&0&&0&&0&&0&&0&&0 \\ &0 && 0 && 0 &&0&&0&&0&& 0 &&0&&0&&0&& 1 &&0&&0&&0 & \ea \right).
\end{equation}

  \item  The Young tableaux  $YB_5$ has a decomposition
\begin{equation*}
  \tableau{1 3}= \emptyset+ \tableau{1 3}
\end{equation*}
By using  the (\ref{bzt}),  the second one  contribute to symbol as follows
\begin{equation}\label{BY32}
\left(\ba{@{}c@{}c@{}c@{}c@{}c@{}c@{}c@{}c@{}c@{}c@{}c@{}c@{}c@{}c@{}c@{}c@{}c@{}c@{}c@{}c@{}c@{}c@{}c@{}c@{}c@{}c@{}c@{}c@{}c@{}} 0&&0&&1&&1&&0&&0&&0&&0&&0&&0&&0&&0&&0&&0&&0 \\ &0 && 0 && 0 &&0&&0&&0&& 0 &&0&&0&&0&& 0 &&0&&0&&0 & \ea \right).
\end{equation}
\end{itemize}
By adding the contributions  of the  five parts  $YB_1, YB_2, YB_3, YB_4$, and $YB_5$,   the  symbol of the partition $\lambda$ is
 \begin{equation}\label{adds}
\left(\ba{@{}c@{}c@{}c@{}c@{}c@{}c@{}c@{}c@{}c@{}c@{}c@{}c@{}c@{}c@{}c@{}c@{}c@{}c@{}c@{}c@{}c@{}c@{}c@{}c@{}c@{}c@{}c@{}c@{}c@{}} 0&&0&&1&&1&&1&&2&&2&&2&&3&&3&&3&&3&&3&&4&&4 \\ &1 && 1 && 1 &&2&&2&&2&& 3&&3&&3&&3&& 4 &&5&&5&&5 & \ea \right)
\end{equation}
which is consistent with the result computed by  Definition \ref{Dn}.

Now, we summary the  general principles of the decomposition (\ref{Gp}). The heights of the two rows of a  pairwise patten in Proposition \ref{Pb} are notated  as $j, j+1$ (Then $j$ is even) .
\begin{itemize}
\item If $j$th and $j+1$th rows  are even,   we will  get a partition  as $YB_2$ after the decomposition.  This partition can be  decomposed into two partitions further. One is a row with width of $n_{j}$ boxes. Since  $l-L(j+1)+\lambda_{L(j+1)}$ is even, by using the  formula (\ref{brb}),  its contribution  to symbol is
     \begin{equation}\label{bevenb}
       \Bigg(\!\!\!\ba{c}0\;\;0\cdots 0 \;\;  0 \cdots  0 \;\;0 \cdots  0 \\
\;\;\;0\cdots 0 \;\; \underbrace{1 \cdots  1}_{n_j/2}\;\underbrace{0 \cdots  0 }_{l(j+1)/2} \ea \Bigg)
     \end{equation}
   with $L(j)=\sum^{m}_{i=j} n_{i}$.
   And another part  of the decomposition   is a rectangle.  The height of the rectangle  is even,  consisting of  $j-2$ boxes.  And the width $W(j)$ depend on the parity of rows before    this pairwise patten
   $$W(j)=n_j+(n_{j-1}+b).$$
 If the row before  the pairwise patten is odd, $b=-1$ as shown in Fig.\ref{S1}, otherwise $b=0$ \ref{S2}. By using the formula (\ref{b22r}), the contribution to symbol of this rectangle is
     \begin{equation*}
       \Bigg(\!\!\!\ba{c}0\;\;0\cdots 0 \;\;  \overbrace{ {(j-2)}/{2} \cdots  {(j-2)}/{2}}^{W(j)} \;\;0 \cdots  0 \\
\;\;\;0\cdots 0 \;\; \underbrace{ {(j-2)}/{2} \cdots  {(j-2)}/{2}}_{W(j)}\;\underbrace{0 \cdots  0 }_{l(j+1)/2} \ea \Bigg)
     \end{equation*}
   \begin{figure}[!ht]
  \begin{center}
    \includegraphics[width=4in]{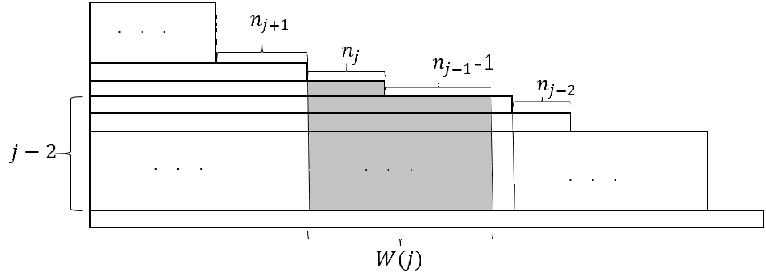}
  \end{center}
  \caption{The $(j+1)$th, $j$th rows  are even and in a pairwise  pattern. The  $(j-1)$th, $(j-2)$th   rows  are old and  in a pairwise  pattern. $n_{j}$  is even and   $n_{j-1}$ is old. Then $W(j)=n_j+(n_{j-1}-1)$.}
  \label{S1}
\end{figure}

    \begin{figure}[!ht]
  \begin{center}
    \includegraphics[width=4in]{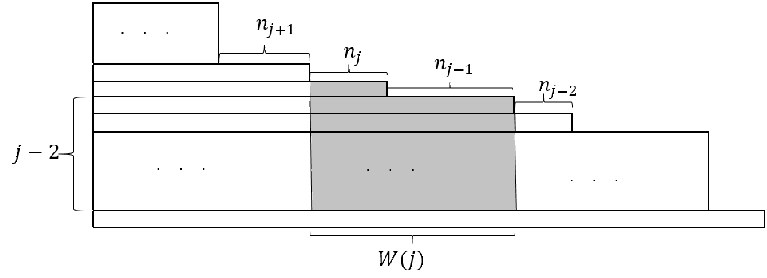}
  \end{center}
  \caption{The $(j+1)$th, $j$th rows  are even and in a pairwise  pattern. The  $(j-1)$th, $(j-2)$th   rows  are even and  in a pairwise  pattern. $n_{j}$,    $n_{j-1}$ are even. Then $W(j)=n_j+n_{j-1}$.}
  \label{S2}
\end{figure}

\item  If these two rows    is odd,   we will  get a partition  as $YB_3$ after the decomposition.  This partition can be  decomposed into two partitions further. One is the partition $21^{n_j}$ .  Since  $l-(L(j+1)+1)+\lambda_{L(j+1)+1}$ is old, by using the formula (\ref{bzt}),  its contribution  to symbol is
     \begin{equation}\label{boddt}
\Bigg(\!\!\!\ba{c}0\;\;0\cdots 0 \;\; \overbrace{ 1 \cdots  1}^{n_j/2+1} \;\;\overbrace{ 0 \cdots  0}^{(l(j+1)-1)/2} \\
\;\;\;0\cdots 0 \;\;{ 0 \cdots  0}\;{0 \cdots  0} \ \ea \Bigg)
     \end{equation}
   And the other  part  of the decomposition   is rectangle. The height of the rectangle is even,   consisting of  $j-2$ boxes.  And the width $W(j)$ depend on the parities  of rows before   this pairwise patten
   $$W(j)=(n_j+1)+(n_{j-1}+b),$$
 If the row before  the pairwise patten is odd,  $b=-1$ as shown in Fig.\ref{S4}, otherwise $b=0$ as shown in Fig.\ref{S3}. By using the formula (\ref{b22r}), the contribution to symbol of this rectangle is
     \begin{equation*}
       \Bigg(\!\!\!\ba{c}0\;\;0\cdots 0 \;\;  \overbrace{ {(j-2)}/{2} \cdots  {(j-2)}/{2}}^{W(j)} \;\;0 \cdots  0 \\
\;\;\;0\cdots 0 \;\; \underbrace{ {(j-2)}/{2} \cdots  {(j-2)}/{2}}_{W(j)}\;\underbrace{0 \cdots  0 }_{(l(j+1)-1)/2} \ea \Bigg).
     \end{equation*}
 \end{itemize}

 \begin{figure}[!ht]
  \begin{center}
    \includegraphics[width=4in]{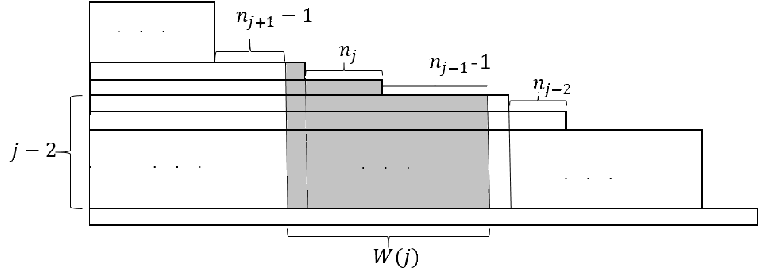}
  \end{center}
  \caption{The $j+1$, $j$th rows  are old and in a pairwise  pattern. The  $(j-1)$th, $(j-2)$th   rows  are old and  in a pairwise  pattern. $n_{j}$   and $n_{j-1}$ are even. Then $W(j)=(n_j+1)+(n_{j-1}-1)$.}
  \label{S4}
\end{figure}

 \begin{figure}[!ht]
  \begin{center}
    \includegraphics[width=4in]{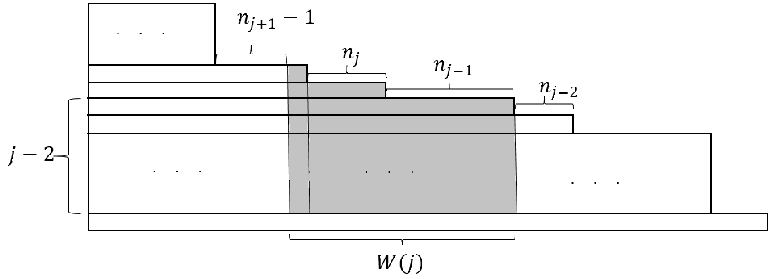}
  \end{center}
  \caption{The $(j+1)$th, $j$th rows  are old and in a pairwise  pattern. The  $(j-1)$th, $(j-2)$th   rows  are even and  in a pairwise  pattern. $n_{j}$  is even and   $n_{j-1}$ is old. Then $W(j)=(n_j+1)+n_{j-1}$.}
  \label{S3}
\end{figure}

Let $m$ be the number of rows of a partition.   $YB_1$ can be seen as a special case with $n_{m+1}=0$ and $n_{m+2}=0$. And  $YB_5$ can be seen as a special case with $j-2=0$.
 The parity of  rows in the pairwise patten is determined by $\Delta(j)=\frac{1-(-1)^{L(j)}}{2}$. If $\Delta(j)=0$, the rows in a pairwise patten are even, otherwise they are odd.  Taking account of the  contribution of the first row in the partition,  we can summary above results as the following closed formula of symbol.
\begin{Pro}\label{Ppb}
For a partition $\lambda=m^{n_m}{(m-1)}^{n_{m-1}}\cdots{1}^{n_1}$ in the $B_n$ theory, denoting  $L(j)=\sum^{m}_{i=j} n_{i}$ and $\Delta(j)=\frac{1-(-1)^{L(j)}}{2}$, we introduce the following notations
$$W(j)=\frac{1}{2}((n_{j}+\Delta(j))+(n_{j-1}-\Delta(j-1)))$$
and
$$Sp(j+1)=\frac{1}{2}(l(j+1)-\Delta(j+1)).$$
 Then the symbol of $\lambda$ is
\begin{eqnarray}\label{Pbi}
  && \sigma_B(\lambda)= \Bigg(\!\!\!\ba{c}\overbrace{0\;\;0\cdots \cdots 0}^{(l(1)+1)/2}  \\
 \;\;\;\underbrace{1\cdots \cdots 1}_{(l(1)-1)/2}\ \ea \Bigg)+\sum^{[(m+1)/2]}_{j=1}\Bigg\{\Bigg(\!\!\!\ba{c}0\;\;0\cdots 0 \;\; \overbrace{ {(j-2)}/{2} \cdots  {(j-2)}/{2}}^{W(j)} \;\;\overbrace{ 0 \cdots  0}^{Sp(j+1)} \\
\;\;\;0\cdots 0 \;\;\underbrace{ {(j-2)}/{2} \cdots  {(j-2)}/{2}}_{W(j)}\;\underbrace{0 \cdots  0}_{Sp(j+1)} \ \ea \Bigg) \nonumber \\
   &&\quad\,\,  + \Bigg(\!\!\!\ba{c}0\;\;0\cdots 0 \;\; \overbrace{ 1 \cdots  1}^{\Delta(j+1)\Delta W(j)} \;\;\overbrace{ 0 \cdots  0}^{Sp(j+1)} \\
\;\;\;0\cdots 0 \;\;\underbrace{ 0 \cdots  0}_{\Delta(j+1)\Delta W(j)}\;\underbrace{0 \cdots  0}_{Sp(j+1)} \ \ea \Bigg)+ \Bigg(\!\!\!\ba{c}0\;\;0\cdots 0 \;\; \overbrace{ 0 \cdots  0}^{(1-\Delta(j+1))\Delta W(j)} \;\;\overbrace{ 0 \cdots  0}^{Sp(j+1)} \\
\;\;\;0\cdots 0 \;\;\underbrace{ 1 \cdots  1}_{(1-\Delta(j+1))\Delta W(j)}\;\underbrace{0 \cdots  0}_{Sp(j+1)} \ \ea \Bigg)\Bigg\}
\end{eqnarray}
with $\Delta W(j)= \frac{1}{2}(n_{j}-\Delta(j))$.
\end{Pro}

\begin{flushleft}
  \textbf{Remark:} The  following partition has an even number of  rows
  \be\label{Yb3}
\tableau{2 5 9 13 15 19}.\non
\ee
According to Proposition \ref{Pb}, the last row is even. The  proposition is also collect in the special case with $n_{m+1}=0$.
\end{flushleft}

\subsection{Symbol of  partitions in the  $C_n$ theory}
We explain our strategy for  computing  symbol by an example.
\begin{flushleft}
\textbf{Example:} $C_{58}$, $\lambda=8^47^26^35^44^43^42^41^2$, the Young tableaux is
\begin{figure}[!ht]
  \begin{center}
    \includegraphics[width=2.4in]{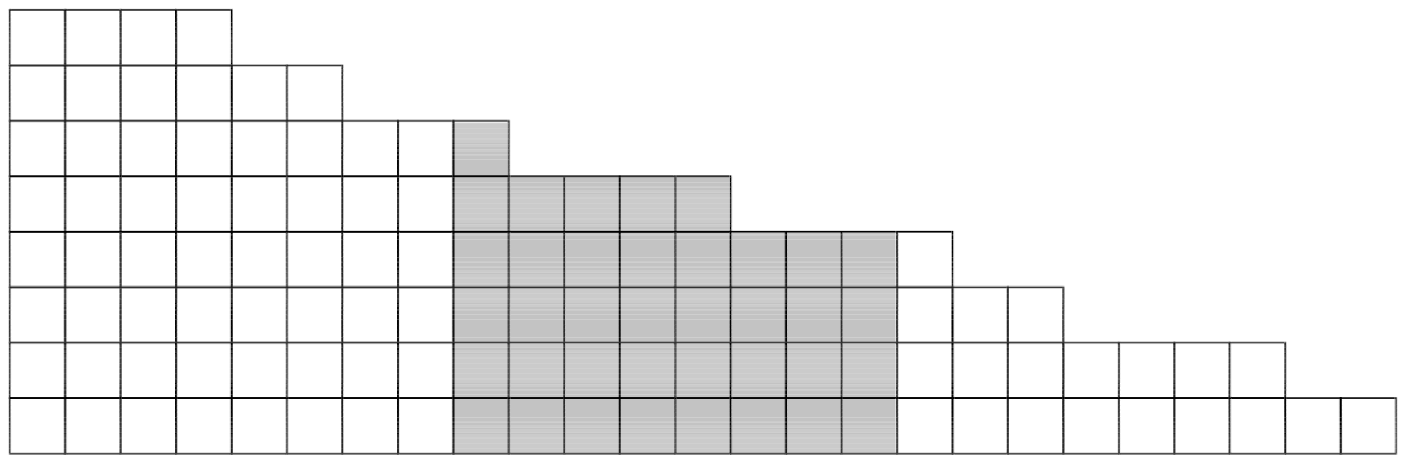}
  \end{center}
\end{figure}
\end{flushleft}

We can decompose it into parts whose contribution to symbol  can be computed by the formulas in  subsection 3.1. The gray part of the Young tableau has the following decomposition.
\begin{equation}\label{Yc}
  \tableau{1 5 8 8 8 8}= \tableau{8 8 8 8}+ \tableau{1 5}
\end{equation}

According to Proposition \ref{Pc}, the longest two rows in a rigid $C_n$ partition both contain either  an even or an odd number  of boxes.  This pairwise pattern then continues. So we need not  to discuss the contribution of the  first row independent as in the $B_n$ case.   For the two rows in a pairwise patten in Proposition \ref{Pc}, we denote the heights as $j, j+1$ (Then $j$ is odd). Note that we should append a extra 0 as the last part of the partition if the first row is odd which means $l\rightarrow l+1$.  We can derive the  formula of the contribution of each block taking the gray part as a example.    Denoting  $L(j)=\sum^{m}_{i=j} n_{i}$.
\begin{itemize}
\item  If the  two rows  in the pairwise patten   is odd,   we will  get a partition  as the Young tableaux in (\ref{Yc}) or as shown in Fig.\ref{Soo}.  This partition can be further decomposed into two partitions. One is the partition $21^{n_j}$($n_j$ is even). Since  $l-(L(j+1)+1)+\lambda_{L(j+1)+1}$ is old, by using the formula (\ref{bzt}),  its contribution  to symbol is
     \begin{equation}\label{coddt}
\Bigg(\!\!\!\ba{c}0\;\;0\cdots 0 \;\; \overbrace{ 1 \cdots  1}^{n_j/2+1} \;\;\overbrace{ 0 \cdots  0}^{(l(j+1)-1)/2} \\
\;\;\;0\cdots 0 \;\;{ 0 \cdots  0}\;{0 \cdots  0} \ \ea \Bigg)
     \end{equation}
   Another part  of the decomposition   is a rectangle. The height of the rectangle is  $j-1$ which is  even.  And the width $W(j)$ depend on the parity of rows before    this pairwise pattern,
   $$W(j)=(n_j+1)+(n_{j-1}+b).$$
 If the row before  the pairwise patten is odd, $b=-1$ as shown in Fig.\ref{Soo}, otherwise $b=0$ as shown in Fig.\ref{Soe}. By using the formula (\ref{b22r}), the contribution to symbol of this rectangle is
     \begin{equation*}
       \Bigg(\!\!\!\ba{c}0\;\;0\cdots 0 \;\;  \overbrace{ {(j-2)}/{2} \cdots  {(j-2)}/{2}}^{W(j)} \;\;0 \cdots  0 \\
\;\;\;0\cdots 0 \;\; \underbrace{ {(j-2)}/{2} \cdots  {(j-2)}/{2}}_{W(j)}\;\underbrace{0 \cdots  0 }_{(l(j+1)-1)/2} \ea \Bigg).
     \end{equation*}
         \begin{figure}[!ht]
  \begin{center}
    \includegraphics[width=4in]{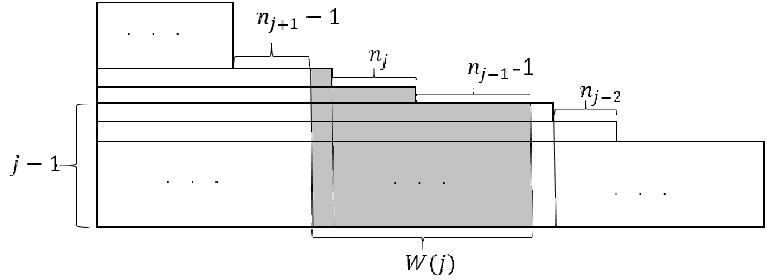}
  \end{center}
  \caption{The $j+1$, $j$th rows  are old and in a pairwise  pattern. The  $(j-1)$th, $(j-2)$th   rows  are old and  in a pairwise  pattern. $n_{j}$   and $n_{j-1}$ are even. Then $W(j)=(n_j+1)+(n_{j-1}-1)$.}
  \label{Soo}
\end{figure}
    \begin{figure}[!ht]
  \begin{center}
    \includegraphics[width=4in]{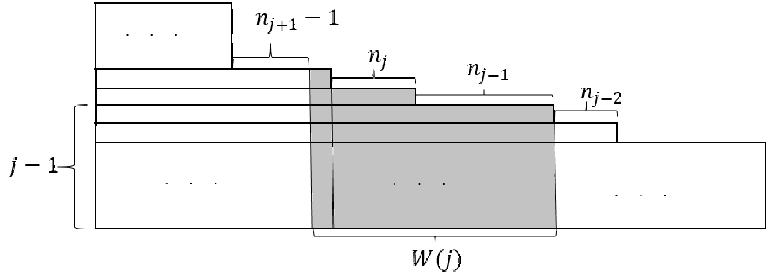}
  \end{center}
  \caption{The $(j+1)$th, $j$th rows  are old and in a pairwise  pattern. The  $(j-1)$th, $(j-2)$th   rows  are even and  in a pairwise  pattern. $n_{j}$  is even and   $n_{j-1}$ is old. Then $W(j)=(n_j+1)+n_{j-1}$.}
  \label{Soe}
\end{figure}

\item If the two rows  in the pairwise patten   is even,   we will  get a partition  which  can be further decomposed into two partitions as in the $B_n$ theory. One is a single row with width of $n_{j}$ boxes. Since   $l-L(j+1)+\lambda_{L(j+1)}$ is even, by using the formula (\ref{brb}),  its contribution  to symbol is
     \begin{equation}\label{cevenb}
       \Bigg(\!\!\!\ba{c}0\;\;0\cdots 0 \;\;  0 \cdots  0 \;\;0 \cdots  0 \\
\;\;\;0\cdots 0 \;\; \underbrace{1 \cdots  1}_{n_j/2}\;\underbrace{0 \cdots  0 }_{l(j+1)/2} \ea \Bigg).
     \end{equation}
   Another one   is a rectangle. The height of the rectangle is  $j-1$ which is even.  And the width $W(j)$ depend on the parity of rows before  and after  this pairwise patten
   $$W(j)=n_j+(n_{j-1}+b).$$
  If the row before  the pairwise patten is odd, Fig.\ref{Seo}, otherwise $b=0$ as shown in Fig.\ref{See}. By using the formula (\ref{b22r}), the contribution to symbol of this rectangle is
     \begin{equation*}
       \Bigg(\!\!\!\ba{c}0\;\;0\cdots 0 \;\;  \overbrace{ {(j-2)}/{2} \cdots  {(j-2)}/{2}}^{W(j)} \;\;0 \cdots  0 \\
\;\;\;0\cdots 0 \;\; \underbrace{ {(j-2)}/{2} \cdots  {(j-2)}/{2}}_{W(j)}\;\underbrace{0 \cdots  0 }_{l(j+1)/2} \ea \Bigg)
     \end{equation*}

 \end{itemize}

 \begin{figure}[!ht]
  \begin{center}
    \includegraphics[width=4in]{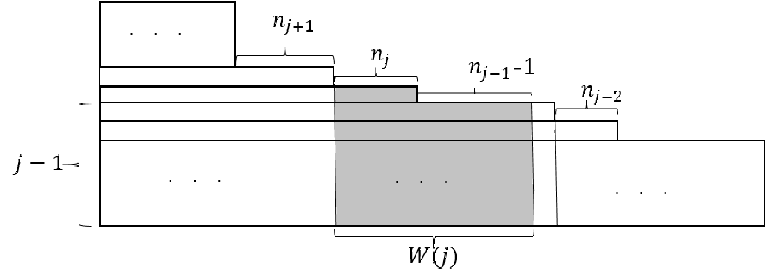}
  \end{center}
  \caption{The $(j+1)$th, $j$th rows  are even and in a pairwise  pattern. The  $(j-1)$th, $(j-2)$th   rows  are old and  in a pairwise  pattern. $n_{j}$  is even and   $n_{j-1}$ is old. Then $W(j)=n_j+(n_{j-1}-1)$.}
  \label{Seo}
\end{figure}
    \begin{figure}[!ht]
  \begin{center}
    \includegraphics[width=4in]{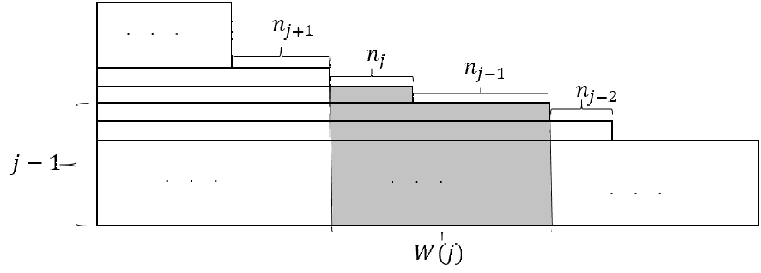}
  \end{center}
  \caption{The $(j+1)$th, $j$th rows  are even and in a pairwise  pattern. The  $(j-1)$th, $(j-2)$th   rows  are even and  in a pairwise  pattern. $n_{j}$,    $n_{j-1}$ are even. Then $W(j)=n_j+n_{j-1}$.}
  \label{See}
\end{figure}

 The parity of the pairwise patten is determined by $\Delta(j)=\frac{1-(-1)^{L(j)}}{2}$. If $\Delta(j)=0$, the rows in the pairwise patten are even otherwise they are odd.  We can summary the above results as the following closed formula of symbol.
\begin{Pro}\label{Ppc}
For a partition $\lambda=m^{n_m}{(m-1)}^{n_{m-1}}\cdots{1}^{n_1}$ in the $C_n$ theory, notating  $L(j)=\sum^{m}_{i=j} n_{i}$ and $\Delta(j)=\frac{1-(-1)^{L(j)}}{2}$, we introduce the following notations
$$W(j)=\frac{1}{2}((n_{j}+\Delta(j))+(n_{j-1}-\Delta(j-1)))$$
and
$$Sp(j+1)=\frac{1}{2}(l(j+1)-\Delta(j+1)).$$
 Then the symbol of the partition is
\begin{eqnarray}\label{Pci}
  && \sigma_C(\lambda)= \sum^{[(m+1)/2]}_{j=1}\Bigg\{\Bigg(\!\!\!\ba{c}0\;\;0\cdots 0 \;\; \overbrace{ {(j-1)}/{2} \cdots  {(j-1)}/{2}}^{W(j)} \;\;\overbrace{ 0 \cdots  0}^{Sp(j+1)} \\
\;\;\;0\cdots 0 \;\;\underbrace{ {(j-1)}/{2} \cdots  {(j-1)}/{2}}_{W(j)}\;\underbrace{0 \cdots  0}_{Sp(j+1)} \ \ea \Bigg) \nonumber \\
   &&\,\, + \Bigg(\!\!\!\ba{c}0\;\;0\cdots 0 \;\; \overbrace{ 1 \cdots  1}^{\Delta(j+1)\Delta W(j)} \;\;\overbrace{ 0 \cdots  0}^{Sp(j+1)} \\
\;\;\;0\cdots 0 \;\;\underbrace{ 0 \cdots  0}_{\Delta(j+1)\Delta W(j)}\;\underbrace{0 \cdots  0}_{Sp(j+1)} \ \ea \Bigg)+ \Bigg(\!\!\!\ba{c}0\;\;0\cdots 0 \;\; \overbrace{ 0 \cdots  0}^{(1-\Delta(j+1))\Delta W(j)} \;\;\overbrace{ 0 \cdots  0}^{Sp(j+1)} \\
\;\;\;0\cdots 0 \;\;\underbrace{ 1 \cdots  1}_{(1-\Delta(j+1))\Delta W(j)}\;\underbrace{0 \cdots  0}_{Sp(j+1)} \ \ea \Bigg)\Bigg\}
\end{eqnarray}
with $\Delta W(j)= \frac{1}{2}(n_{j}-\Delta(j))$.
\end{Pro}
\begin{flushleft}
  \textbf{Remarks:}
\begin{itemize}
  \item Compared  with  the above formula, the first term in  the formula (\ref{Pbi}) disappear. The reason is that   the first row is in a pairwise patten in the $C_n$ theroy.
  \item  Note that formula (\ref{coddt}) is exact match with the formula (\ref{boddt}) and the formula (\ref{cevenb}) is exact match with the formula (\ref{bevenb}), which will be  discussed in the last section.
\end{itemize}
\end{flushleft}

\subsection{Symbol of  partitions in the $D_n$ theory}
According to  Proposition \ref{Pd}, the first row of a partition in the $D_n$ theory is even. For the first step to compute the symbol,  an extra 0  is appended as the last part of the partition.  Then the  length of the new partition is odd. The Young tableaux of the new partition  can be seen as a  Young tableaux  in the $B_n$ theory. According to  Remark 2 of Definition \ref{Dn},  the remaining  steps to compute symbol are all consistent with those in  the $B_n$ theory.  So we can  use the results  in the $B_n$ theory directly  except the first row.  The  first row  appended an extra  0   is different from that of a partition in the $B_n$ theory. However, we can calculate the contribution to symbol of the first row directly
  $$\Bigg(\!\!\!\!\ba{c}\overbrace{\;1 \;\; 1\cdots 1\;\; 1 \cdots 1\;}^{l/2} \\
\;\;0\;\;0\cdots 0\; \;0\cdots 0 \ \ea \Bigg)$$
where $l$ is the length of the  partition. From above discussions, the contributions of other parts of the partition are the same with those of the  $B_n$ theory. So we have
\begin{Pro}\label{Ppd}
For a partition $\lambda=m^{n_m}{(m-1)}^{n_{m-1}}\cdots{1}^{n_1}$ in the $D_n$ theory, denoting  $L(j)=\sum^{m}_{i=j} n_{i}$ and $\Delta(j)=\frac{1-(-1)^{L(j)}}{2}$, we introduce the  following notations
$$W(j)=\frac{1}{2}((n_{j}+\Delta(j))+(n_{j-1}-\Delta(j-1)))$$
and
$$Sp(j+1)=\frac{1}{2}(L(j+1)-\Delta(j+1)).$$
 Then the symbol of $\lambda$ is
\begin{eqnarray}\label{Pdi}
  && \sigma_D(\lambda)= \Bigg(\!\!\!\!\ba{c}\overbrace{\;1 \;\; 1\cdots 1\;\; 1 \cdots 1\;}^{l/2} \\
\;\;0\;\;0\cdots 0\; \;0\cdots 0 \ \ea \Bigg)+\sum^{[(m+1)/2]}_{j=1}\Bigg\{\Bigg(\!\!\!\ba{c}0\;\;0\cdots 0 \;\; \overbrace{ {(j-2)}/{2} \cdots  {(j-2)}/{2}}^{W(j)} \;\;\overbrace{ 0 \cdots  0}^{Sp(j+1)} \\
\;\;\;0\cdots 0 \;\;\underbrace{ {(j-2)}/{2} \cdots  {(j-2)}/{2}}_{W(j)}\;\underbrace{0 \cdots  0}_{Sp(j+1)} \ \ea \Bigg) \nonumber \\
   &&\quad\,\,  + \Bigg(\!\!\!\ba{c}0\;\;0\cdots 0 \;\; \overbrace{ 1 \cdots  1}^{\Delta(j+1)\Delta W(j)} \;\;\overbrace{ 0 \cdots  0}^{Sp(j+1)} \\
\;\;\;0\cdots 0 \;\;\underbrace{ 0 \cdots  0}_{\Delta(j+1)\Delta W(j)}\;\underbrace{0 \cdots  0}_{Sp(j+1)} \ \ea \Bigg)+ \Bigg(\!\!\!\ba{c}0\;\;0\cdots 0 \;\; \overbrace{ 0 \cdots  0}^{(1-\Delta(j+1))\Delta W(j)} \;\;\overbrace{ 0 \cdots  0}^{Sp(j+1)} \\
\;\;\;0\cdots 0 \;\;\underbrace{ 1 \cdots  1}_{(1-\Delta(j+1))\Delta W(j)}\;\underbrace{0 \cdots  0}_{Sp(j+1)} \ \ea \Bigg)\Bigg\}
\end{eqnarray}
with $\Delta W(j)= \frac{1}{2}(n_{j}-\Delta(j))$.
\end{Pro}

\section{Discussions}
In this section,  a comparison  of between the closed formula of symbol in  this paper   and the  one  proposed  in \cite{Shou 2:2016} is made.   The   closed formula  found in \cite{Shou 2:2016} is
\begin{Pro}[\cite{Shou 2:2016}]\label{Fbcd}
For a partition $\lambda=m^{n_m}{(m-1)}^{n_{m-1}}\cdots{1}^{n_1}$, we introduce two notations
$$\Delta^{T}_i=\frac{1}{2}(\sum^{m}_{k=i}n_k+\frac{1+(-1)^{i+1}}{2}),\quad\quad  P^{T}_i=\frac{1+\pi_i}{2}$$
where the superscript $T$ indicates it is related to the top row of the symbol and
$$\pi_i=(-1)^{\sum^{m}_{k=i}n_k}\cdot(-1)^{i+1+t},$$
for $B_n(t=-1)$, $C_n(t=0)$, and $D_n(t=1)$ theories.
Other parallel notations
$$\Delta^{B}_i=\frac{1}{2}(\sum^{m}_{k=i}n_k+\frac{1+(-1)^{i}}{2}),\quad \quad  P^{B}_i=\frac{1-\pi_i}{2}$$
where the superscript $B$ indicates it is  related to the bottom row of the symbol and $P^{B}_i$ is a projection operator similar to  $P^{T}_i$.
Then the symbol $\sigma(\lambda)$   is
\begin{equation}\label{FbFb}
\sigma(\lambda)=\sum^m_{i=1}\big\{ \Bigg(\!\!\!\ba{c}0\;\;0\cdots 0\;\; \overbrace{1\cdots 1}^{P^{T}_i \Delta^{T}_i}  \\
\;\;\;\underbrace{0\cdots 0 \;\;0 \cdots 0}_{l+t}\ \ea \Bigg)
+
 \Bigg(\!\!\!\ba{c}\overbrace{0\;\;0\cdots 0\;\; 0\cdots 0}^{l}  \\
\;\;\;0\cdots 0 \;\;\underbrace{1\cdots 1}_{P^{B}_i \Delta^{B}_i}\ \ea \Bigg)\big\}
\end{equation}
with  $l=(m+(1-(-1)^m)/2)/2$.
\end{Pro}

In fact, this formula is equivalent to  the computational rules of symbol in the following  table \cite{Shou 2:2016}. We calculate the symbol of a rigid partition by summing the contribution of each row according to this table.
\begin{center}
\begin{tabular}{|c|c|c|c|}\hline
\multicolumn{4}{|c|}{ Contribution to  symbol of the $i$\,th row  }\\ \hline
Parity of row & Parity of $i+t+1$ & Contribution   & $L$  \\ \hline
odd & even  & $\Bigg(\!\!\!\ba{c}0 \;\; 0\cdots \overbrace{ 1\;\; 1\cdots1}^{L} \\
\;\;\;0\cdots 0\;\; 0\cdots 0 \ \ea \Bigg)$ & $\frac{1}{2}(\sum^{m}_{k=i}n_k+1)$  \\ \hline
even & odd  & $\Bigg(\!\!\!\ba{c}0 \;\; 0\cdots \overbrace{ 1\;\; 1\cdots1}^{L} \\
\;\;\;0\cdots 0\;\; 0\cdots 0 \ \ea \Bigg)$   & $\frac{1}{2}(\sum^{m}_{k=i}n_k)$  \\ \hline
even & even  &  $\Bigg(\!\!\!\ba{c}0 \;\; 0\cdots 0\;\; 0 \cdots 0 \\
\;\;\;0\cdots \underbrace{1 \;\;1\cdots 1}_{L} \ \ea \Bigg)$   & $\frac{1}{2}(\sum^{m}_{k=i}n_k)$  \\ \hline
odd & odd  &  $\Bigg(\!\!\!\ba{c}0 \;\; 0\cdots 0\;\; 0 \cdots 0 \\
\;\;\;0\cdots \underbrace{1\; \;1\cdots 1}_{L} \ \ea \Bigg)$     & $\frac{1}{2}(\sum^{m}_{k=i}n_k-1)$  \\ \hline
\end{tabular}
\end{center}

We prove the  equivalence of the old and  new closed formulas of symbol.
\begin{Pro}\label{Dual}
For a partition in the $B_n$ theory, the symbol given by the formula (\ref{Pbi}) is equal to the symbol given by the  formula (\ref{FbFb}).
\end{Pro}
  \begin{proof} First,  we  decompose a rigid partition into three parts: the first row,  rows in a pairwise patten,  and the last even row  not in a pairwise patten. Then we  calculate   the contribution to symbol of each part  using  different closed formulas.
  \begin{itemize}
    \item   For the first row, according to the computational rules in the above table,  its contribution  to symbol is the same with the first term of the  closed formula (\ref{Pbi}).
    \item  \begin{enumerate}
    \item For the two even rows in a pairwise patten with partition $2^{2b}1^{2c}$, according to the above table,  its contribution to symbol is
\be\label{ceven}
 \Bigg(\!\!\!\ba{c}0\;\;0\cdots 0 \;\; 0\cdots 0\;\; \overbrace{ 1\cdots 1}^{b} \\
\;\;\;0\cdots 0 \;\; \underbrace{1\cdots 1\;1\cdots 1}_{b+c} \ \ea \Bigg).
\ee

Next we compute the contribution   of $2^{2b}1^{2c}$ using the new closed formula (\ref{Pbi}).  We can decompose $2^{2b}1^{2c}$ into two parts $2^{2b}$ and $1^{2c}$. For the parts $2^{2b}$, using the formula (\ref{b22})  $b$ times  its  contribute to symbol is
 \be\label{b22c}
 \Bigg(\!\!\!\ba{c}0\;\;0\cdots 0 \;\; 0\cdots 0\;\; \overbrace{ 1\cdots 1}^{b} \\
\;\;\;0\cdots 0 \;\; \underbrace{0\cdots 0\;1\cdots 1}_{b+c} \ \ea \Bigg).
\ee
    According to the formula (\ref{bevenb}),  the contribution of the parts $1^{2c}$ is
    \be\label{bodd}
 \Bigg(\!\!\!\ba{c}0\;\;0\cdots 0 \;\; 0\cdots 0\;\; \overbrace{ 0\cdots 0}^{b}\\
\;\;\;0\cdots 0 \;\; \underbrace{1\cdots 1\;0\cdots 0}_{b+c} \ \ea \Bigg).
\ee
Combining the formulas (\ref{b22c}) and (\ref{bodd}), we get
\be
 \Bigg(\!\!\!\ba{c}0\;\;0\cdots 0 \;\; 0\cdots 0\;\; \overbrace{ 1\cdots 1}^{b} \non \\
\;\;\;0\cdots 0 \;\; \underbrace{1\cdots 1\;1\cdots 1}_{b+c} \ \ea \Bigg)\non
\ee
which is consistent with the formula (\ref{ceven}).
    \item For the two odd rows in a pairwise patten with partition $2^{2b+1}1^{2c}$, according to the above table,  its contribution to symbol is
\be\label{codd}
 \Bigg(\!\!\!\ba{c}0\;\;0\cdots 0 \;\; \overbrace{1\cdots 1\;\; 1\cdots1}^{b+c+1} \\
\;\;\;0\cdots 0 \;\;0\cdots0\;\underbrace{1\cdots 1}_{b+1} \ \ea \Bigg).
\ee

 Next we compute the contribution  of the parts $2^{2b+1}1^{2c}$ using the new closed formula  (\ref{Pbi}). We can decompose $2^{2b+1}1^{2c}$ into two parts $2^{2b+1}$ and $1^{2c}$. For  the parts $2^{2b+1}$,  using  the formula (\ref{b22}) $b+1$ times   its   contribution  to symbol is
 \be\label{b22o}
 \Bigg(\!\!\!\ba{c}0\;\;0\cdots 0 \;\;\overbrace{ 0\cdots 0\;\;  1\cdots 1}^{b+c+1} \non \\
\;\;\;0\cdots 0 \;\; 0\cdots 0\;\underbrace{1\cdots 1}_{b+1} \ \ea \Bigg).\non
\ee
    According to the formula (\ref{boddt}), the contribution of the parts $1^{2c}$  is
    \be\label{bevenbb}
 \Bigg(\!\!\!\ba{c}0\;\;0\cdots 0 \;\;\overbrace{ 1\cdots 1\;\;  0\cdots 0}^{b+c+1} \non \\
\;\;\;0\cdots 0 \;\; 0\cdots 0\;\underbrace{0\cdots 0}_{b+1} \ \ea \Bigg).\non
\ee
Combining the formulas (\ref{b22o}) and (\ref{bevenbb}), we get
\be
 \Bigg(\!\!\!\ba{c}0\;\;0\cdots 0 \;\; \overbrace{1\cdots 1\;\; 1\cdots1}^{b+c+1}\non \\
\;\;\;0\cdots 0 \;\;0\cdots0\;\underbrace{1\cdots 1}_{b+1} \ \ea \Bigg).\non
\ee
which is consistent with the  formula (\ref{codd}).

    \end{enumerate}

\item The lase even row  which is not in a pairwise patten can be seen as a special case of  a pairwise patten.  According to the above discussions,  we get the same result by different closed formulas.
  \end{itemize}
\end{proof}
\begin{flushleft}
  \textbf{Remark:} We can  prove the theorem  for the  $C_n$ and  $D_n$ theories similarly.
\end{flushleft}

Although the old  closed  formula of symbol (\ref{FbFb}) is  concise, the new one (\ref{Pbi}) is more essential, which is derived  from two basic building blocks \textbf{'Rule A'} and \textbf{'Rule B'}. Using  the new closed formula,  we can prove the  following proposition proposed in \cite{Shou 2:2016} neatly.
\begin{Pro}[\cite{Shou 2:2016}]\label{Dual}
The symbol of  a partition in  $B_n,C_n$, and $D_n$ theories is
 \be
    \boxed{\left(\ba{@{}c@{}c@{}c@{}c@{}c@{}c@{}c@{}c@{}c@{}} \alpha_1 &&  \alpha_2 &&\cdots  &&  \alpha_m  \\ & \beta_1 && \cdots && \beta_{m+t} && & \ea \right)}\non
\ee
where  $m=(l+(1-(-1)^l)/2)/2$, $l$ is the length of the partition. $t=-1$ for the $B_n$ theory,  $t=0$  for the $C_n$ theory, and  $t=1$ for the $D_n$ theory.
For a partition in the $B_n$ theory or a partition with only old  rows in  the $C_n$ theory, we have $\alpha_i\leq \beta_{i+t}$.  For a partition in $D_n$  theory or a partition with only even  rows in the $C_n$ theory, we have $\alpha_i\geq \beta_{i+t}$.
\end{Pro}

  \begin{proof} We discuss the contributions  to each term of  the   pair $(\alpha_i, \beta_{i+t})$ by each term in the closed formula  (\ref{Pbi}).
  \begin{itemize}
                    \item For a partition in the $B_n$ theory,   the second term in the formula (\ref{Pbi}) contribute to $\alpha_i$ and $ \beta_{i+t}$ equality.  The third and the fourth terms contribute  to $\alpha_i$ or $ \beta_{i+t}$  at most one.  While   the first term contribute one to $ \beta_{i+t}$.  So we  draw  the conclusion.

                    \item For a partition with even rows only in the $C_n$ theory,   the second term in the formula (\ref{Pci}) will not appear. So we draw the conclusion.  For a partition with  only odd rows,  the third term in the formula (\ref{Pci}) will not appear. So we draw the conclusion.

                    \item For a partition in the $D_n$ theory,  the second term in the formula (\ref{Pdi}) contribute to $\alpha_i$ and $ \beta_{i+t}$ equality.  The third and the fourth terms contribute to $ \beta_{i+t}$  at most one. While   the first term contribute  to each  $\alpha_i$ one.  So we draw the conclusion.

                  \end{itemize}
\end{proof}

\section{Acknowledgments}
We would like to thank Ming Huang  and Xiaobo Zhuang for  many helpful discussions.  This work of B.Shou was supported by a grant from  the Postdoctoral Foundation of Zhejiang Province. The work of Qiao Wu has been supported by Key Research Center of Philosophy and Social Science of Zhejiang Province: Modern Port Service Industry and Creative Culture Research Center(15JDLG02YB).




\end{document}